\documentclass[english,11pt]{amsart}

\usepackage[english]{babel}
\usepackage[T1]{fontenc}
\usepackage{lmodern}

\usepackage{MnSymbol}
\usepackage[centering]{geometry}

\theoremstyle{plain}
\newtheorem{theo}{Theorem}
\newtheorem{lemm}[theo]{Lemma}
\newtheorem{prop}[theo]{Proposition}

\theoremstyle{definition}

\theoremstyle{remark}
\newtheorem{rema}[theo]{Remark}

\usepackage{microtype}
\usepackage{enumerate}
\usepackage{graphicx}
\usepackage{color}
\usepackage{bm}
\usepackage{mathtools}
\usepackage[dvipsnames]{xcolor}
\definecolor{FlatRed}{RGB}{231,76,60}
\definecolor{FlatGreen}{RGB}{46,204,113}
\definecolor{FlatBlue}{RGB}{52,152,219}
\definecolor{FlatYellow}{RGB}{241,196,15}
\colorlet{FlatViolet}{FlatRed!50!FlatBlue}
\colorlet{FlatBrown}{FlatRed!50!FlatGreen}
\colorlet{FlatOrange}{FlatRed!50!FlatYellow}
\colorlet{FlatCyan}{FlatGreen!50!FlatBlue}

\usepackage{enumitem}
\usepackage{textcomp}
\usepackage[normalem]{ulem}

\usepackage{hyperref}
\hypersetup{
    colorlinks,
    linkcolor={FlatRed},
    citecolor={FlatGreen},
    urlcolor={FlatBlue}
}

\title[Quotients of Gaussian multiplicative chaos]{Tail profile of bulk Gaussian multiplicative chaos measures I: bulk/boundary quotients}

\author{Yichao Huang}
\address{Beijing Institute of Technology, School of Mathematics and Statistics, Beijing, China}
\email{yichao.huang@bit.edu.cn}

\begin{document}

\begin{abstract}
This is the first part of a series of papers devoted to studying the right tail profile of a bulk Gaussian multiplicative chaos measure with uniform singularity on the boundary. We investigate the bulk/boundary quotients of Gaussian multiplicative chaos measures appearing in boundary Liouville conformal field theory, for which we establish preliminary joint moment bounds. These moment bounds will be a crucial ingredient in establishing the right tail profile of the bulk Gaussian multiplicative chaos measure in subsequent papers. The main idea is to implement the so-called localization trick at the boundary, and we also record a useful generalization of Kahane's convexity inequality, which is of independent interest. The study of the universal tail profiles of general bulk measures and bulk/boundary quotients as well as connections to integrability results of boundary Liouville conformal field theory will be pursued in subsequent papers.
\end{abstract}

\maketitle

\section{Introduction}

Consider the bulk Gaussian multiplicative chaos measure that appeared in the study of boundary Liouville conformal field theory~\cite{huang2018liouville}. Using the half-plane model $\mathbb{H}=\{x+iy\in\mathbb{C}~;~y>0\}$ as our convention throughout this paper, we consider the Gaussian multiplicative chaos measure with parameter $\gamma\in(0,2)$, defined on some small neighborhood of the origin $0\in\overline{\mathbb{H}}$ by assigning the following mass to open measurable sets $A$ in this neighborhood:
\begin{equation}\label{eq:DefinitionBulkMeasure}
    \mu^{\mathrm{H}}(A)=\int_{A}\text{Im}(z)^{-\frac{\gamma^2}{2}}dM_{\mathrm{N}}(z)
\end{equation}
where $dM_{\mathrm{N}}$ is the classical Gaussian multiplicative chaos measure (see Section~\ref{subse:gaussian_multiplicative_chaos}) associated to the covariance kernel with Neumann boundary condition
\begin{equation*}
    K_{\mathrm{N}}(z,w)=-\ln|z-w||z-\overline{w}|+O(1).
\end{equation*}
This type of random measure has a strong singularity along the boundary $\mathbb{R}=\partial\mathbb{H}$, as both the background measure $\text{Im}(z)^{-\frac{\gamma^2}{2}}$ and the covariance kernel $K_{\mathrm{N}}$ of the underlying Gaussian field $X$ blows up near $\mathbb{R}$. We refer to $\mu_{\mathrm{H}}$ as the bulk Gaussian multiplicative chaos measure in the sequel, see Section~\ref{subse:gaussian_multiplicative_chaos_with_boundary_singularity} below for a brief review.

It was established in~\cite{Huang:2023aa} the following exact threshold for the finiteness of positive moments of the measure $\mu_{\mathrm{H}}$ near the boundary $\mathbb{R}=\partial\mathbb{H}$ (the partial case with $\gamma<\sqrt{2}$ was settled in~\cite{huang2018liouville}). Especially, for any Carleson cube $Q_r=[-r,r]\times [0,2r]$ near the origin with small enough $r$,
\begin{equation}\label{eq:Huang23Result}
    \mathbb{E}\left[\mu^{\mathrm{H}}(Q_r)^{p}\right]<\infty
\end{equation}
if and only if $p<\frac{2}{\gamma^2}$. In contrast to this result, the exact threshold for the finiteness of positive moments of the classical two-dimensional Gaussian multiplicative chaos measure is $\frac{4}{\gamma^2}$, see~\cite[Theorem~2.11]{Rhodes_2014}.

In this paper, we establish preliminary joint moment bounds for the bulk/boundary quotients of Gaussian multiplicative chaos measures. The (one-dimensional) boundary Gaussian multiplicative chaos measure corresponding to the above bulk Gaussian multiplicative chaos measure, according to conventions in boundary Liouville conformal field theory~\cite{huang2018liouville}, is defined formally as
\begin{equation*}
    \mu^{\partial}(I)=\int_{I}e^{\frac{\gamma}{2}X(w)-\frac{\gamma^2}{8}\mathbb{E}[X(w)^2]}dw,
\end{equation*}
for any interval $I\subset\mathbb{R}$ near the origin, where the log-correlated Gaussian field $X$ is defined on $I$ with the same covariance $K_{\mathrm{N}}(z,w)$ above (but restricted to the boundary $\mathbb{R}=\partial\mathbb{H}$). We are interested in establishing basic properties about expressions of the following type with parameters $(p,q)\in\mathbb{R}$:
\begin{equation}\label{eq:p-q_Quotient}
    \frac{\mu^{\mathrm{H}}(Q)^{p}}{\mu^{\partial}(I)^{q}},
\end{equation}
where in most cases, $Q$ will be a cube with sides parallel to the $xy$-axes, and $I$ is the projection of $Q$ onto $\mathbb{R}$. Especially in this note, we are interested in the finiteness of expectation of such expressions above, which we will refer to as the joint $(p,q)-$moments for bulk/boundary quotient of Gaussian multiplicative chaos measures.

Indeed, in the probabilistic representation of boundary Liouville conformal field theory by the Gaussian multiplicative chaos method, one studies the joint distribution of these measures $(\mu^{\mathrm{H}},\mu^{\partial})$, and quotient terms of this type appear in various formulas (see e.g.~\cite[Section~3.6]{huang2018liouville}). In a subsequent paper, we will show that the so-called right tail profile constant (also known as the reflection coefficient for the boundary Liouville conformal field theory in dimension two) of the bulk Gaussian multiplicative chaos measure is related to the above bulk/boundary quotient, especially with $(p,q)=(\frac{2}{\gamma^2},1)$. Notice that this case is not covered by the result of~\cite{Huang:2023aa} on the finiteness~\eqref{eq:Huang23Result}. The original motivation of this paper was to establish sufficient conditions of the finiteness of the expectation of the above bulk/boundary quotients with $q=1$ (that is Theorem~\ref{th:main_result}), in order to justify the well-definedness and the expression of the tail profile constant for the bulk measure $\mu^{\mathrm{H}}$. In the course of writing out the details, we discovered an alternative approach that yields sufficient conditions (which we conjecture to be optimal as well) for the finiteness of the expectation of a general class of $(p,q)$-bulk/boundary quotients. On the one hand, we hope that the simplicity of our result, as well as the heuristics we give during the presentation, will be pedagogical for the reader to understand some interesting properties and interactions of the bulk/boundary measures; on the other hand, we think that these kinds of conditions on the bulk/boundary quotients are of general interest, and have other potential applications to boundary Liouville conformal field theory (see Section~\ref{sec:discussions_and_perspectives}).

\subsection{Main results of this paper}
The main motivation of writing this note is to establish the following positive moment bound for the above bulk/boundary quotients of Gaussian multiplicative chaos measures, as a preliminary result to subsequent papers in this series. We call a subset $Q$ of $\overline{\mathbb{H}}$ a Carleson cube if $Q=[a,b]\times [0,b-a]$ for some $a<b$.
\begin{theo}[Finiteness of moments for bulk/boundary quotients, case $q=1$]\label{th:main_result}
Suppose that $Q$ is a Carleson cube of the upper half-plane near the origin, say $[-r,r]\times[0,r]\subset\overline{\mathbb{H}}$ for some small $r>0$, and $I$ is its intersection with $\mathbb{R}$ (in this case $[-r,r]\subset\mathbb{R}$). With the notations in the previous section and $q=1$, the $(p,q)$-th moment of bulk/boundary quotient of Gaussian multiplicative chaos measures is finite, i.e.
\begin{equation*}
    \mathbb{E}\left[\frac{\mu^{\mathrm{H}}(Q)^{p}}{\mu^{\partial}(I)}\right]<\infty
\end{equation*}
if $p<\min(\frac{2}{\gamma^2}+\frac{1}{2},\frac{4}{\gamma^2})$. In particular, $p$ can be strictly greater than $\frac{2}{\gamma^2}$ with this condition.
\end{theo}

The method that we use to establish Theorem~\ref{th:main_result} in the sequel is quite robust, and we can use it to upgrade the above theorem to a general class of $(p,q)$-joint moments with optimal conditions. Of course, Theorem~\ref{th:main_result} (resp. \eqref{eq:Huang23Result}, the main result of~\cite{Huang:2023aa}) is contained in Theorem~\ref{th:main_result_positive} as special cases $q=1$ (resp. $q=0$), but we decide to single out Theorem~\ref{th:main_result} in view of the applications we have in mind. We will only treat the case for $p,q\geq 0$ since this will be the most relevant case for most applications (also because other cases are either easier or similar).
\begin{theo}[Finiteness of joint moments for bulk/boundary quotients]\label{th:main_result_positive}
Suppose that $Q$ is a Carleson cube of the upper half-plane $\overline{\mathbb{H}}$ near the origin, and $I$ its intersection with $\mathbb{R}$. Suppose that $p,q\geq 0$. Then
\begin{equation*}
    \mathbb{E}\left[\frac{\mu^{\mathrm{H}}(Q)^{p}}{\mu^{\partial}(I)^{q}}\right]<\infty
\end{equation*}
if $p<\min(\frac{2}{\gamma^2}+\frac{q}{2},\frac{4}{\gamma^2})$.
\end{theo}
We conjecture that the above moment bound is optimal (the optimality when the threshold is $\frac{4}{\gamma^2}$ is shown in Lemma~\ref{lemm:finiteness_of_quotients_with_positive_distance}), and that the joint moment explodes at the critical threshold $p_c=\min(\frac{2}{\gamma^2}+\frac{1}{2},\frac{4}{\gamma^2})$, but this requires some further refinement of our method. We would like to obtain some proof of the explosion of joint moments with a soft method (using similar multifractal analysis considerations as in this note), but another possible route to proving the optimality of the threshold and the explosion at criticality is by studying the tail profile problem as we will do in this series of notes, see Section~\ref{sec:discussions_and_perspectives} for discussions and perspectives in this direction.

These theorems serve as a first step towards a better understanding of the joint law of the bulk/boundary measures. These aspects will be pursued separately. To the best of the author's knowledge, the question of a systematic study of bulk/boundary quotient using the Gaussian multiplicative chaos method has not yet been initiated in the literature, and we propose a quite soft yet general strategy with some preliminary estimates in the current paper.

Let us at once explain the bounds appearing in Theorem~\ref{th:main_result_positive} in a very heuristic way. One can look at the scaling exponents of the bulk/boundary quotients in the same manner as the bulk measure of~\cite{Huang:2023aa}. It turns out (see Section~\ref{subse:exact_scaling_relations} below) that when $q<2p$, the boundary scaling exponents of the $(p,q)$-boundary quotient is the same as that of the $(p-\frac{q}{2})$-th moment of the bulk measure $\mu^{\mathrm{H}}$. It is a classical observation in branching processes and multiplicative cascades/chaos theories that sign changes in the scaling exponents will ultimately lead to thresholds such at the moment bounds: this explains the factor $\frac{2}{\gamma^2}+\frac{q}{2}$ appearing in the threshold of Theorem~\ref{th:main_result_positive} via~\eqref{eq:Huang23Result}. The other factor $\frac{4}{\gamma^2}$ is the moment bound of the bulk measure with some positive distance away from the boundary (i.e. the classical moment bound for the two-dimensional Gaussian multiplicative chaos measure). The phenomenon here is that near the boundary, the nominator and the denominator in the bulk/boundary quotient have non-trivial cancellation, and very informally speaking, the boundary measure behaves (as reflected in the scaling exponents) like the ``square root'' of the bulk measure near the boundary when both measures are very large. This intuition is the guiding philosophy behind our proofs, and is the main reason why we think that the bounds in Theorem~\ref{th:main_result_positive} are optimal.

We stress that Theorem~\ref{th:main_result_positive} covers the case where $q=2p$: this case is somewhat critical in that the scaling exponents of the (localized or unlocalized) bulk/boundary quotients in Sections~\ref{subse:exact_scaling_relations} and~\ref{subse:exact_scaling_relations_for_bulk_boundary_quotients_with_localization_at_the_boundary} are trivially constant equal to $0$. Therefore, a straightforward scaling analysis seems ineffective. This case has a clear physical relevance and is raised as a question to us by Baptiste Cerclé because of its relation to exact formulas (or integrability conjectures) on the boundary Liouville conformal field theory. We completely settle this moment bound problem in Section~\ref{sub:finiteness_of_the_joint_moment_case_with_large_enough_q_} (independently of our original motivation on the reflection coefficient of the bulk measure) which serves as a small preparation in this direction, see Section~\ref{sec:discussions_and_perspectives} for some future questions.

\subsection{Structure of this paper}
The main message of this series of papers is that, when applying the localization trick (a technique developed in~\cite{Rhodes_2019}, see Section~\ref{sec:phenomenology_and_strategy_of_proofs}) to study the right tail of a given Gaussian multiplicative chaos measure, it is helpful to construct an auxiliary Gaussian multiplicative chaos measure that captures the singularities of the original measure, and apply the localization trick to the auxiliary measure and reduce the problem to a quotient type problem treated in Theorems~\ref{th:main_result} or \ref{th:main_result_positive}. The current note, which illustrates this idea in arguably the simplest setting, is organized in the following order:

In Section~\ref{subse:mathematical_background_on_gaussian_multiplicative_chaos_measures}, we provide the necessary mathematical background for the different Gaussian multiplicative chaos measures that we study in this paper. Especially, we review the classical construction of the classical Gaussian multiplicative chaos measures, then the bulk Gaussian multiplicative chaos measures with boundary singularity. The latter is motivated by the so-called boundary Liouville conformal field theory, and is going to be our prime example.

In Section~\ref{append:variants_of_kahane_s_convexity_inequality}, we record a non-convex version of Kahane's convexity inequality, as the original convex form of Kahane's convexity inequality is quite burdensome when applied directly to bulk/boundary quotients of exponentials of Gaussian fields in our applications. This generalization will be used throughout this series of papers: it is essentially a modified Gaussian comparison principle (the conditions are exactly those that appear in the classical Slepian's lemma for Gaussian processes) and allows one to reduce many estimates to some special case, e.g. the so-called exact scaling case or the star-scale invariant case.

In Section~\ref{sec:phenomenology_and_strategy_of_proofs}, we provide an overview of our proof strategy. The key idea is an induction scheme from above, based on successive applications of the localization trick~\cite{Rhodes_2019} at the boundary, building on the observations of~\cite{Huang:2023aa}. The induction scheme is somewhat reminiscent of~\cite{kahane1976certaines}, but the localization trick for the Gaussian multiplicative chaos measure at the boundary is new and essential to our method.

In Section~\ref{sec:preliminary_estimates_on_the_bulk_boundary_quotient}, we prove the joint moment bounds for the bulk/boundary quotient (i.e. Theorems~\ref{th:main_result} and \ref{th:main_result_positive}). It is not clear from existing literature why such a term should be finite, and our results yield precise critical thresholds. Variants and consequences of these results will be useful in subsequent papers, especially for establishing the right tail profile of the bulk measure.

In Section~\ref{sec:discussions_and_perspectives}, we collect various observations and discuss some possible generalizations of our findings. Several open questions for further investigation are also proposed.

\subsection*{Acknowledgements.} We thank Mo Dick Wong for very useful communications during the workshop ``Random Interacting Systems, Scaling Limits, and Universality'' at the National University of Singapore. We thank Baptiste Cerclé for asking us about the bulk/boundary quotient moment bound in the critical case $q=2p$, which reinvigorated our interest in this problem. Y.H. is partially supported by National Key R\&D Program of China (No. 2022YFA1006300) and NSFC-12301164.

\section{Mathematical background on Gaussian multiplicative chaos measures}\label{subse:mathematical_background_on_gaussian_multiplicative_chaos_measures}

We will be primarily working with the upper half-plane $\mathbb{H}=\{x+iy\in\mathbb{C}~;~y>0\}$ model with boundary $\partial\mathbb{H}=\mathbb{R}$. We call a subset $Q$ of $\overline{\mathbb{H}}$ a Carleson cube of the upper-half plane if $Q=[a,b]\times [0,b-a]$ for some $a<b$.

\subsection{Gaussian multiplicative chaos}\label{subse:gaussian_multiplicative_chaos}
Consider a log-correlated Gaussian field $X$ on a bounded domain $\Omega\subset\mathbb{R}^d$, that is, $X$ is formally a Gaussian process with covariance kernel
\begin{equation}\label{eq:GeneralCovarianceKernel}
    \mathbb{E}[X(z)X(w)]=-\ln|z-w|+g(z,w),
\end{equation}
where $g:\Omega^2\to\mathbb{R}$ is a symmetric smooth bounded function.\footnote{The regularity assumption on $g$ can be relaxed, but our focus is not in this direction and we refer to~\cite{junnila2019decompositions} for discussions on this aspect.}

By formal Gaussian process, we mean that the random field $X$ is not defined pointwise (as formally $\mathbb{E}[X^2(z)]=\infty$ due to the log-singularity), but any smooth positive compactly supported mollifier $\rho$ defined on $\mathbb{R}^{d}$ will be able to make $X$ a well-defined Gaussian process by standard convolution
\begin{equation*}
    X_{\epsilon}\coloneqq \rho_{\epsilon}\ast X.
\end{equation*}

The (classical) Gaussian multiplicative chaos with parameter $\gamma>0$, starting with the seminal work of Kahane~\cite{kahane1985chaos}, is defined as the following limit in probability of random measures on $\Omega$ defined by declaring that for all Borel sets $A\subset\mathbb{R}^{d}$,
\begin{equation*}
    M_{\gamma}(A)\coloneqq \lim_{\epsilon\to 0}\int_{A}e^{\gamma X_{\epsilon}(z)-\frac{\gamma^2}{2}\mathbb{E}[X_{\epsilon}(z)^2]}\sigma(dz)
\end{equation*}
where $\sigma(dz)$ is absolutely continuous with respect to the Lebesgue measure $d\lambda(z)$ on $\mathbb{R}^d$.\footnote{One can relax this by replacing the absolutely continuity with some dimension capacity condition~\cite{Berestycki_2017}, but this is not our focus and we will not work with the most general setting.} Here, the convergence of measures is in the weak-$\star$ sense. Kahane established the following phase transition phenomenon: the measure $M_{\gamma}$ is a.s. non-trivial if and only $\gamma\in(0,\sqrt{2d})$, and is a.s. degenerate (i.e. $M_{\gamma}\equiv 0$) if $\gamma\geq \sqrt{2d}$. There are actually different renormalizations for the critical case $\gamma=\sqrt{2d}$~\cite{Duplantier_2014,Duplantier2014,powell2020critical,Lacoin24AIHP} so that we get non-trivial and interesting limiting measures, the interested readers can consult the above references. It should be mentioned that all results in this paper should extend in some natural sense to the critical and supercritical cases (see~\cite{Rhodes_2014} for a brief introduction to the supercritical Gaussian multiplicative chaos), but one should use different techniques, e.g. see~\cite{Wong:2019aa} for the classical universal tail profile on the critical Gaussian multiplicative chaos. We do not plan to investigate beyond the subcritical case in this paper.

There are various different proofs and improvements of Kahane's fundamental result (historically, Kahane only showed the convergence in law) in the literature, we refer to~\cite{Berestycki_2017} for a quick and simple approach, and to~\cite{Rhodes_2014,Berestycki:2024aa} for more general information. We will not collect further basic facts about the classical Gaussian multiplicative chaos measures here; we will recall them when needed and give pointers to the above mentioned reviews for their proofs. We also do not discuss about the connections of Gaussian multiplicative chaos to Liouville conformal field theory in the sense of~\cite{david2016liouville,Kupiainen_2020,guillarmou2020conformal} (as we don't use these techniques in this paper) in this paper, and refer the readers to the review papers~\cite{vargas2017lecture,Guillarmou:2024aa}.

\begin{rema}
Although the rigorous treatment of the Gaussian multiplicative chaos measure always requires passing by a regularization procedure (with for instance the mollifier approach above), this will be mostly a technical issue which is quite standard in our setting. We will thus omit details about this regularization and refer the readers to~\cite{robert2010gaussian,Berestycki:2024aa,huang2018liouville,Huang:2023aa} for justifications, and we will choose to write down directly unregularized versions in the estimates below with pointers to similar proofs in the existing literature, in order not to distract the readers from the key ideas of the proof.
\end{rema}

\subsection{Gaussian multiplicative chaos with boundary singularity}\label{subse:gaussian_multiplicative_chaos_with_boundary_singularity}
In this paper, we will be primarily considering the following variant of the above classical Gaussian multiplicative chaos, motivated by the definition of the boundary Liouville conformal field theory, see~\cite{huang2018liouville}. For convenience, we work with the half-plane setting.

Consider a small cube-neighborhood $Q_r=[-r,r]\times[0,r]\subset\overline{\mathbb{H}}$ of the origin $0$ in $\overline{\mathbb{H}}$, where $r$ is chosen to be small enough so that the following covariance kernel with smooth function $g$:
\begin{equation}\label{eq:GeneralCovarianceKernelHalfPlane}
    K_\mathrm{N}(z,w)=-\ln|z-w||z-\overline{w}|+g(z,w)
\end{equation}
is positive definite on $Q_r$. When $g(z,w)\equiv 0$, we call this the canonical or exact-scaling covariance kernel (of the upper half-plane model):
\begin{equation}\label{eq:ExactScalingKernel}
    K_\mathrm{C}(z,w)=-\ln|z-w||z-\overline{w}|.
\end{equation}

Denote by $I_r=[-r,r]\subset\mathbb{R}$ the boundary of $Q_r$, and notice that $K_\mathrm{N}$ restricted to the boundary is $-2\ln|z-w|+O(1)$, with an extra multiplicative factor $2$ due to the boundary effect. Consider then the log-correlated Gaussian field $X$ on $Q_r$ with covariance kernel $K_\mathrm{N}$, that is,
\begin{equation*}
    \mathbb{E}[X(z)X(w)]=K_\mathrm{N}(z,w).
\end{equation*}
One can then construct the Gaussian multiplicative chaos measure associated to the Gaussian field $X$ in the classical sense in the interior of $\mathbb{H}$, and we denote this measure by $M_{\mathrm{N}}$ (we omit the parameter $\gamma$). Now the type of Gaussian multiplicative chaos we will be considering on $\mathbb{H}$ is defined by
\begin{equation}
    d\mu^{\mathrm{H}}(z)=\text{Im}(z)^{-\frac{\gamma^2}{2}}dM_{\mathrm{N}}(z).
\end{equation}
The singularity in the background metric is in some sense uniform around the boundary $\mathbb{R}=\partial\mathbb{H}$, and its origin should be traced back to the different renormalization conventions in boundary Liouville conformal field theory and classical Gaussian multipliative chaos theory. We will not explain this in detail and take it as a definition, and refer to~\cite[Section 2]{huang2018liouville} for calculations and further explanations.

The main goal of this paper is to study some information about the joint law of the bulk Gaussian multiplicative chaos above, with the following boundary Gaussian multiplicative chaos defined with the same covariance kernel (but restricted to the boundary):
\begin{equation*}
    \mu^{\partial}(I)\coloneqq \lim_{\epsilon\to 0}\int_{I}e^{\frac{\gamma}{2} X_{\epsilon}(z)-\frac{\gamma^2}{8}\mathbb{E}[X_{\epsilon}(z)^2]}dz
\end{equation*}
for Borel sets $I\subset\mathbb{R}$. We point out at once that this is a classical one-dimensional Gaussian multiplicative chaos in the sense of Section~\ref{subse:gaussian_multiplicative_chaos} above, but the parameter in the exponential is $\frac{\gamma}{2}$ and the covariance is $-2\ln|z-w|+O(1)$ with an extra multiplicative factor of $2$.

\subsection{Kahane's convexity inequality}\label{subse:kahane_s_convexity_inequality}

A useful toolbox in the study of Gaussian multiplicative chaos measures is the comparison principle known as Kahane's convexity inequality. We take the statement from~\cite[Theorem~3.18]{Berestycki:2024aa} (with weaker assumption on the measure $\sigma$ using our convention in this article) since it is written directly in the continuum; we also refer to this book for the proof.
\begin{lemm}[Kahane's convexity inequality]\label{lemm:KahaneConvexityInequality}
Let $\Omega\subset\mathbb{R}^{d}$ be a bounded domain, and $X,Y$ are log-correlated Gaussian fields on $\Omega$ with the following covariance comparison inequality:
\begin{equation*}
    \forall x,y\in\Omega,\quad \mathbb{E}[X(x)X(y)]\leq\mathbb{E}[Y(x)Y(y)].
\end{equation*}
Then for any convex functional $F:\mathbb{R}_{\geq 0}\to\mathbb{R}$ with at most polynomial growth at the boundaries,
\begin{equation*}
    \mathbb{E}\left[F\left(\int_{\Omega}e^{X(z)-\frac{1}{2}\mathbb{E}[X(z)^2]}d\sigma(z)\right)\right]\leq\mathbb{E}\left[F\left(\int_{\Omega}e^{Y(z)-\frac{1}{2}\mathbb{E}[Y(z)^2]}d\sigma(z)\right)\right],
\end{equation*}
where $\sigma$ is absolutely continuous with respect to the Lebesgue measure on $\Omega$.

\end{lemm}
Alternatively, in Section~\ref{append:variants_of_kahane_s_convexity_inequality}, a proof of a generalized version of this inequality (with the convexity condition relaxed) is proved in the finite vector setting: generalizations to the continuum is similar to~\cite[Theorem~3.18]{Berestycki:2024aa}.

Rougly speaking, this comparison inequality allows us to trade the correction $g(z,w)$ in the covariance kernel of type~\eqref{eq:GeneralCovarianceKernel} or \eqref{eq:GeneralCovarianceKernelHalfPlane} for some global multiplicative constant in most estimates, thus reducing the study to the case with e.g. the so-called exact scaling (or pure-log) covariance kernel of~\eqref{eq:ExactScalingKernel}. However, this is not directly applicable to the bulk/boundary quotient problem that we are considering (as the quotient functional is not convex in the sense of the above lemma), or that using the original Kahane's convexity inequality requires some long preparations and modifications to meet the convexity requirement. In light of these obstructions, we improve Kahane's convexity inequality in Section~\ref{append:variants_of_kahane_s_convexity_inequality} by relaxing the convexity assumption.

\section{A non-convex variant of Kahane's inequality}\label{append:variants_of_kahane_s_convexity_inequality}
In this preliminary section, we establish a comparison inequality in the spirit of Kahane's convexity inequality (reviewed in Section~\ref{subse:kahane_s_convexity_inequality}), but where the convexity assumption is relaxed. In fact, the inspiration comes from reading a lecture note of~\cite[Proposition~5.1]{Biskup:2020aa}, where the so-called Slepian's inequality for Gaussian processes is renamed Kahane's inequality, which raises the question of whether the original condition of Kahane's convexity inequality can be replaced by the conditions in the classical Slepian's inequality for Gaussian processes. Curiously, we cannot find this particular variant of Kahane's convexity inequality written down in the existing literature, so we write down the statement and provide a quick proof here (which also covers the original Kahane's convexity inequality).

Let us first describe the class of inequalities of Gaussian processes that we are interested in. The original Kahane's convexity inequality has the following distinctive traits:
\begin{enumerate}
    \item It is within the family of Gaussian comparison inequalities, as the underlying process is a Gaussian field and where the basic assumption is a domination condition of covariances, see Lemma~\ref{lemm:KahaneConvexityInequality}.
    \item The basic variables in the inequality are the normalized exponentials, i.e. elements of form $e^{Z_i-\frac{1}{2}\mathbb{E}[Z_i^2]}$ with $\bm{Z}=\{Z_i\}_{i\in I}$ a centered Gaussian vector (or Gaussian field when $I$ is assumed separable). In particular, the $\mathbb{E}[Z_i^2]$ term in the exponential, serving as a renormalization constant, distinguishes Kahane's convexity inequality from other classical Gaussian comparison inequalities, e.g. Slepian's inequality or Sudakov-Fernique's inequality.
    \item The inequality is about the expected value of some \emph{convex} functional $F$ acting on a positive linear combination of the normalized exponentials $e^{Z_i-\frac{1}{2}\mathbb{E}[Z_i^2]}$.
\end{enumerate}

The following theorem relaxes the convexity condition of the functional in the third item above. We will state this for finite Gaussian vectors as we feel that this setting illustrates better the main idea behind this generalization, but the proof naturally generalizes to continuous Gaussian processes in the spirit of~\cite[Theorem~3.18]{Berestycki:2024aa}.

\subsection{Gaussian coupling lemma}
Let us start by setting up the stage for the Gaussian coupling technique before the statement of the theorem.

Let $\bm{A},\bm{B}$ be two centered Gaussian vectors independent between them (that is, $\mathbb{E}[A_iB_j]=0$ for all $i,j$) and define for $0\leq t\leq 1$,
\begin{equation}\label{eq:DefintionZQ}
    Z_i(t)\coloneqq \sqrt{1-t}A_i+\sqrt{t}B_i;\quad W_i(t)=\exp\left(Z_i(t)-\frac{1}{2}\mathbb{E}[Z_i(t)^2]\right)
\end{equation}
to be the Gaussian interpolation (and the renormalized exponential thereof) between $\bm{A},\bm{B}$. Then if $'$ denotes the derivative with respect to the interpolation parameter $t$, we have
\begin{equation*}
    W'_i(t)=(Z'_i(t)-\mathbb{E}[Z_i(t)Z'_i(t)])W_i(t)
\end{equation*}
and 
\begin{equation*}
    \mathbb{E}[Z_i(t)Z'_j(t)]=\frac{1}{2}(\mathbb{E}[B_iB_j]-\mathbb{E}[A_iA_j]).
\end{equation*}

\begin{lemm}
If $G$ is any smooth functional with subgaussian growth at infinity as well as in its first and second derivatives, and if we define
\begin{equation*}
    \varphi(t)\coloneqq\mathbb{E}[G(W_1(t),\dots,W_n(t))],
\end{equation*}
then
\begin{equation}\label{eq:KeyIdentity}
    \varphi'(t)=\frac{1}{2}\sum_{1\leq i,j\leq n}\mathbb{E}[\partial_{ij}G(W_1(t),\dots,W_n(t))W_i(t)W_j(t)](\mathbb{E}[B_iB_j]-\mathbb{E}[A_iA_j]).
\end{equation}
\end{lemm}

Formula~\eqref{eq:KeyIdentity} will be the key identity throughout this section.
\begin{proof}
Deriving with respect to $t$, with the chain rule we get
\begin{equation*}
    \varphi'(t)=\sum\limits_{i=1}^{n}\mathbb{E}[W'_i(t)\cdot \partial_i G(W_1(t),\dots,W_n(t))].
\end{equation*}
Since
\begin{equation*}
    W'_i(t)=(Z'_i(t)-\mathbb{E}[Z_i(t)Z'_i(t)])W_i(t),
\end{equation*}
we have
\begin{equation}\label{eq:IntermediateStep}
    \varphi'(t)=\sum\limits_{i=1}^{n}\mathbb{E}[(Z'_i(t)-\mathbb{E}[Z_i(t)Z'_i(t)])W_i(t)\cdot \partial_i G(W_1(t),\dots,W_n(t))].
\end{equation}

We now fix $1\leq i\leq n$ and apply the Gaussian integration by parts formula to the term $Z'_i(t)$. When performing the Gaussian integration by parts, the Gaussian variable $Z'_i(t)$ should contract with all other Gaussians in the expression above.

$\bullet$ When the other Gaussian variable in the Gaussian integration by parts formula is $W_i(t)$, for each fixed index $1\leq i\leq n$ we get the following contribution:
\begin{equation*}
    \mathbb{E}[W_i(t)\cdot \partial_i G(W_1(t),\dots,W_n(t))]\mathbb{E}[Z'_i(t)Z_i(t)].
\end{equation*}
This cancels with the linear part in $\varphi'(t)$ with fixed $i$, namely
\begin{equation*}
    \mathbb{E}[-\mathbb{E}[Z_i(t)Z'_i(t)]W_i(t)\cdot \partial_i G(W_1(t),\dots,W_n(t))].
\end{equation*}

$\bullet$ It remains to consider the case where the other Gaussian variable in the Gaussian integration by parts formula is in $G(W_1(t),\dots,W_n(t))$. For fixed $1\leq i\leq n$, their contribution amounts to
\begin{equation*}
    \sum\limits_{j=1}^{n}\mathbb{E}[W_i(t)W_j(t)\partial_{ij} G(W_1(t),\dots,W_n(t))]\mathbb{E}[Z'_i(t)Z_j(t)].
\end{equation*}
Recalling that 
\begin{equation*}
    \mathbb{E}[Z_i(t)Z'_j(t)]=\frac{1}{2}(\mathbb{E}[B_iB_j]-\mathbb{E}[A_iA_j])
\end{equation*}
and summing over all $1\leq i\leq n$, we have proven that
\begin{equation*}
    \varphi'(t)=\frac{1}{2}\sum\limits_{1\leq i,j\leq n}\mathbb{E}[\partial_{ij}G(W_1(t),\dots,W_n(t))W_i(t)W_j(t)](\mathbb{E}[B_iB_j]-\mathbb{E}[A_iA_j]).
\end{equation*}
This finished the proof of~\eqref{eq:KeyIdentity}.
\end{proof}

\subsection{Generalized Kahane's inequality for normalized exponentials of Gaussian fields}
We are now ready to announce the non-convex generalization of Kahane's inequality.
\begin{theo}[Non-convex Kahane's inequality for normalized exponentials of Gaussian fields]\label{th:GeneralKahane}
Let $T=\{1,\dots,n\}$ be a discrete index space and $\textbf{A},\textbf{B}$ two centered Gaussian vectors indexed by $T$. Let $G(x_1,\dots,x_n)$ be some real-valued function with subgaussian growth at infinity as well as in all its first and second partial derivatives, and such that for any $(i,j)\in T^2$, we have
\begin{equation}\label{eq:SignCondition}
    (\mathbb{E}[B_iB_j]-\mathbb{E}[A_iA_j])\partial_{ij}G(x_1,\dots,x_n)\geq 0.
\end{equation}
Then with 
\begin{equation*}
    W^{A}_i\coloneqq W_i(0)=\exp\left(A_i-\frac{1}{2}\mathbb{E}[A_i^2]\right),\quad W^{B}_i\coloneqq W_i(1)=\exp\left(B_i-\frac{1}{2}\mathbb{E}[B_i^2]\right),
\end{equation*}
we have
\begin{equation}\label{eq:ConclusionKahane}
    \mathbb{E}[G(W^{A}_1,\dots,W^{A}_n)]\leq\mathbb{E}[G(W^{B}_1,\dots,W^{B}_n)].
\end{equation}
\end{theo}

\begin{proof}
We recall identity~\eqref{eq:KeyIdentity}. When the sign condition~\eqref{eq:SignCondition} is satisfied, for any $(i,j)\in T^2$, the term
\begin{equation*}
    \mathbb{E}[\partial_{ij}G(W_1(t),\dots,W_n(t))W_i(t)W_j(t)](\mathbb{E}[B_iB_j]-\mathbb{E}[A_iA_j])
\end{equation*}
is positive since $W_i(t),W_j(t)$ are positive quantities. Summing up over all $(i,j)\in T^2$ yields the proof of the theorem.
\end{proof}
Notice that this covers the original Kahane's convexity inequality of Lemma~\ref{lemm:KahaneConvexityInequality} in the finite Gaussian vector setting. Generalization of this theorem to the continuous setting (similar to the formulation of Lemma~\ref{lemm:KahaneConvexityInequality}) follows from either~\cite{Rhodes_2014} or~\cite[Theorem~3.18]{Berestycki:2024aa}, and the proof is omitted here.

Let us illustrate two applications of this new form of Kahane's inequality, which simplifies the proofs over the classcial Kahane's convexity inequality and which plays a central role for the general study of bulk/boundary quotients.

\begin{lemm}[Decorrelating bulk/boundary quotients: an upper bound]\label{lemm:DecorrelateBulkBoundaryPositiveDistance}
Suppose that $Q$ is a bounded region in $\overline{\mathbb{H}}$ and $I\subset\mathbb{R}$ is a bounded interval on the boundary. Consider the log-correlated Gaussian field $X$ defined with the general covariance kernel~\eqref{eq:GeneralCovarianceKernelHalfPlane} with smooth bounded correction term $g$. Suppose also that for some $p,q>0$,
\begin{equation*}
    \mathbb{E}[\mu^{\mathrm{H}}(Q)^{p}]<\infty\quad\text{and}\quad\mathbb{E}[\mu^{\partial}(I)^{-q}]<\infty.
\end{equation*}
Then for some constant $C$ depending only on $\gamma,Q,I$ and $||g||_{\infty}$,
\begin{equation*}
    \mathbb{E}\left[\frac{\mu^{\mathrm{H}}(Q)^{p}}{\mu^{\partial}(I)^{q}}\right]\leq C\mathbb{E}[\mu^{\mathrm{H}}(Q)^{p}]\mathbb{E}[\mu^{\partial}(I)^{-q}].
\end{equation*}
\end{lemm}
Notice that this bound is better than what H\"older's inequality can give.

\begin{proof}
By the boundedness assumptions on $Q,I$ and $g$, we know that there exists some constant $K>0$ depending only on $Q,I$ and $||g||_{\infty}$ such that
\begin{equation*}
    \forall x\in Q, \forall y\in I,\quad\mathbb{E}[X(x)X(y)]\geq -K.
\end{equation*}
Therefore, the field $\widetilde{X}=X+\sqrt{K}N$, where $N$ is an independent standard Gaussian variable, has positive covariance on $Q\times I$. Consequently, if $\widetilde{X_Q}$ and $\widetilde{X_I}$ are independent copies of $\widetilde{X}$ restricted to $Q$ and $I$, we have that
\begin{equation*}
    \forall x\in Q, \forall y\in I,\quad\mathbb{E}[\widetilde{X_Q}(x)\widetilde{X_I}(y)]\leq\mathbb{E}[\widetilde{X}(x)\widetilde{X}(y)]
\end{equation*}
while all the other covariances on $(Q\cup I)^2$ are equal. Applying Theorem~\ref{th:GeneralKahane} yields
\begin{equation*}
    \mathbb{E}\left[\frac{\widetilde{\mu^{\mathrm{H}}}(Q)^{p}}{\widetilde{\mu^{\partial}}(I)^{q}}\right]\leq \mathbb{E}[\widetilde{\mu_Q^{\mathrm{H}}}(Q)^{p}]\mathbb{E}[\widetilde{\mu^{\partial}_I}(I)^{-q}],
\end{equation*}
where $\widetilde{\mu^{\mathrm{H}}}$ and $\widetilde{\mu^{\partial}}$ are the respective bulk and boundary Gaussian multiplicative chaos measures associated to the field $\widetilde{X}$, and $\widetilde{\mu_Q^{\mathrm{H}}}$ and $\widetilde{\mu^{\partial}_I}$ are the respective bulk and boundary Gaussian multiplicative chaos measures associated to the fields $\widetilde{X_Q}$ and $\widetilde{X_I}$. Factorizing out the independent Gaussian factors in all the measures on the left and right hand side in the above display yields the result.
\end{proof}

\begin{lemm}[Comparison of joint moments of bulk/boundary quotient for general covariance kernels]\label{lemm:FinitenessBulkBoundaryGeneralKernel}
Suppose that $Q$ is a bounded region in $\overline{\mathbb{H}}$ and $I\subset\mathbb{R}$ is a bounded interval on the boundary. Consider the log-correlated Gaussian field $X$ defined with the general covariance kernel~\eqref{eq:GeneralCovarianceKernelHalfPlane} with smooth bounded correction term $g$, and the exact scale-invariant log-correlated Gaussian field $X_\mathrm{C}$ with covariance~\eqref{eq:ExactScalingKernel} without the correction term. Then for any $p>1,q>0$, the finiteness of the following expectations are equivalent (in the sense that they differ at most by some positive multiplicative constant)
\begin{equation*}
    \mathbb{E}\left[\frac{\mu^{\mathrm{H}}(Q)^{p}}{\mu^{\partial}(I)^{q}}\right]\quad\text{and}\quad\mathbb{E}\left[\frac{\mu^{\mathrm{H}}_{\mathrm{C}}(Q)^{p}}{\mu^{\partial}_{\mathrm{C}}(I)^{q}}\right],
\end{equation*}
where $\mu^{\mathrm{H}}_{\mathrm{C}}$ and $\mu^{\partial}_{\mathrm{C}}$ are the respective bulk and boundary Gaussian multiplicative chaos measure associated to the exact scaling field $X_{\mathrm{C}}$.

In fact, these two expectations are comparable up to a multiplicative constant $C$ that only depends on $\gamma,Q,I$ and $||g||_{\infty}$.
\end{lemm}
The requirement that $p>1$ can be relaxed to $p>0$ and the same conclusion holds: this only changes some signs in the proof below (namely the comparison of kernels on $Q\times Q$ because of the sign change in the second derivative of the function $x\mapsto x^{p}$) and the proof can be modified accordingly.

\begin{proof}
The proof is similar to the one above: we will modify the log-correlated Gaussian fields by adding independent global Gaussian constants.

Consider the following upper bounds of the differences in covariance between $X$ and $X_{\mathrm{C}}$ on different domains:
\begin{equation*}
\begin{split}
    K_Q&=\sup\{z,w\in Q~;~|K(z,w)-K_{\mathrm{C}}(z,w)|\};\\
    K_I&=\sup\{z,w\in I~;~|K(z,w)-K_{\mathrm{C}}(z,w)|\};\\
    K&=\sup\{z,w\in Q\cup I~;~|K(z,w)-K_{\mathrm{C}}(z,w)|\}.
\end{split}
\end{equation*}
Notice that these constants are finite by our boundedness assumptions, and only depend on $Q,I$ and $||g||_{\infty}$.

In view of applying the generalized Kahane's inequality of Theorem~\ref{th:GeneralKahane}, we should make the covariance of $X$ smaller than that of $X_{\mathrm{C}}$ on $Q\times Q$ and $I\times I$, but greater than that of $X_{\mathrm{C}}$ on $Q\times I$. Take then three independent standard Gaussian random variables (also independent of $X$ and $X_{\mathrm{C}}$) denoted $N_Q,N_I,N$ and modify the fields $X$ and $X_{\mathrm{C}}$ in the following manner. Define
\begin{equation*}
\begin{split}
    \widetilde{X}(z)&=X(z)+\sqrt{K}N\mathbf{1}_{\{z\in Q\cup I\}},\\
    \widehat{X_\mathrm{C}}(z)&=X_{\mathrm{C}}(z)+\sqrt{K+K_Q}N_Q\mathbf{1}_{\{z\in Q\}}+\sqrt{K+K_I}N_I\mathbf{1}_{\{z\in I\}}.
\end{split}
\end{equation*}
One verifies that $\widetilde{X}$ and $\widehat{X_{\mathrm{C}}}$ satisfies the comparison of covariances we describe above. Then applying Theorem~\ref{th:GeneralKahane} yields
\begin{equation*}
    \mathbb{E}\left[\frac{\widetilde{\mu^{\mathrm{H}}}(Q)^{p}}{\widetilde{\mu^{\partial}}(I)^{q}}\right]\leq\mathbb{E}\left[\frac{\widehat{\mu^{\mathrm{H}}_{\mathrm{C}}}(Q)^{p}}{\widehat{\mu^{\partial}_{\mathrm{C}}}(I)^{q}}\right],
\end{equation*}
where $\widetilde{\mu^{\mathrm{H}}},\widetilde{\mu^{\partial}}$ are associated to $\widetilde{X}$ and $\widehat{\mu^{\mathrm{H}}_{\mathrm{C}}},\widehat{\mu^{\partial}_{\mathrm{C}}}$ are associated to $\widehat{X_{\mathrm{C}}}$. Factorizing out the independent global Gaussian constant (notice that this only depends on $\gamma,Q,I$ and $||g||_{\infty}$) yields one side of the inequality; the other side is similar.
\end{proof}

These examples are illustrations of a more general class of methods: examining the proof of Theorem~\ref{th:GeneralKahane} yields an interpolation principle in the sense of~\cite[Corollary~3]{wong2020universal}, a key tool in studying the tail profile problem of Gaussian multiplicative chaos measures.

\section{Phenomenology and strategy of proofs}\label{sec:phenomenology_and_strategy_of_proofs}
We have already given some intuition about the bounds in our main Theorems~\ref{th:main_result}, \ref{th:main_result_positive} in the introduction. Indeed, the previous article~\cite{Huang:2023aa} was a rigorous implementation of this intuition, by developing the so-called Sokoban lemmas which are in some sense converse to Kahane's convexity inequality of Lemma~\ref{lemm:KahaneConvexityInequality}, and this allows us to upgrade moment bounds on the Gaussian multiplicative cascades models to results on the Gaussian multiplicative chaos models.

However, working out the corresponding Sokoban lemmas in the context of bulk/boundary quotients is non-trivial when $p$ and $q$ are of the same sign, and we adopt a new strategy inspired by the study of the right tail of the bulk measure $\mu^{\mathrm{H}}$. The strategy is called \emph{localization trick at the boundary}. While the localization trick in the context of Gaussian multiplicative chaos measures is not new (it was initiated in~\cite{Rhodes_2019}), our application of it has a tiny touch of novelty, especially in the case of the right tail of the bulk measure. Indeed, it turns out that the correct way to study the right tail of $\mu^{\mathrm{H}}$ is to not apply the localization trick to the bulk measure we are studying, but rather introduce an auxiliary measure at the boundary (namely $\mu^{\partial}$) and apply the localization to this auxiliary measure. We refer to the next paper in this series for more discussion, especially on the drawback of the classical localization scheme for the bulk measure.

In this paper where we deal with quotients of bulk/boundary measures, the situation is somewhat more direct, since we already have the presence of a boundary measure in our problem (in the above paragraph, the boundary measure does not appear in the problem, and we introduced it out of considerations coming from the boundary Liouville conformal field theory). After recalling some properties of the bulk measure, we explain how to relate the study of the $(p,q)$-joint moment to that of the $(p,q+1)$-joint moment of the bulk/boundary quotient of Gaussian multiplicative chaos measures (with the example of $q=1$, which is the content of Theorem~\ref{th:main_result}).

\subsection{Some observations on the bulk Gaussian multiplicative measure}
First, let us recall some observations in~\cite{huang2018liouville,Huang:2023aa} about the bulk measure $\mu^{\mathrm{H}}$. In the proof of~\eqref{eq:Huang23Result}, we used a Whitney decomposition at the boundary for the upper-half plane model. This roughly means that
\begin{enumerate}
    \item Unlike the classical two-dimensional Gaussian multiplicative cascades/chaos, where we divide a large square into $2^{2}$ equal parts to define the branching structure, the branching structure for $\mu^{\mathrm{H}}$ only comes with $2^{1}$ descendants.
    \item Therefore, the mass of $\mu^{\mathrm{H}}$ is highly concentrated on the boundary $\mathbb{R}$ where the singularity lives, and $\mu^{\mathrm{H}}$ behaves more like a one-dimensional (but in some cases, supercritical) Gaussian multiplicative chaos measure.
\end{enumerate}
Indeed, comparing~\eqref{eq:Huang23Result} with the usual $p_c=\frac{4}{\gamma^2}$ moment bound for two-dimensional Gaussian multiplicative chaos shows that $\mu^{\mathrm{H}}(A)$ for any region $A$ with positive distance to $\mathbb{R}$ inside a Carleson cube $Q$ will not influence the right tail of $\mu^{\mathrm{H}}(Q)$. The fact that the $\mu^{\mathrm{H}}$ behaves like a one-dimensional Gaussian multiplicative chaos is further confirmed in the calculations of the scaling laws for the bulk/boundary quotients in Section~\ref{subse:exact_scaling_relations}. The (very informal) comparison of $\mu^{\mathrm{H}}$ to a supercritical Gaussian multiplicative chaos measure comes from the fact that the moment bound in~\eqref{eq:Huang23Result} can go strictly below $1$, and for classical Gaussian multiplicative chaos this only happens in the supercritical phase (but with a different renormalization). To summarize, it is crucial to perform the analysis at the boundary $\mathbb{R}$ to study the right tail of the bulk measure $\mu^{\mathrm{H}}$.

\subsection{From $q$ to $q+1$}
Now we illustrate the localization trick at the boundary with the example of Theorem~\ref{th:main_result} in the case $q=1$. Fix a small $r>0$ and recall that we denote by $Q_r=[-r,r]\times[0,2r]$. Consider the sequence of Carleson cubes $(Q_{2^{-n}r})_{n\geq 0}$ and denote by $\Pi_{2^{-n}r}=Q_{2^{-n+1}r}\setminus Q_{2^{-n}r}$ the $\Pi$-shaped part between two consecutive cubes. Then (ignoring boundaries of different $\Pi$-shaped regions), the sequence $(\Pi_{2^{-n}r})_{n\geq 1}$ forms a partition of $Q_r$. See Figure~\ref{fig:Pi_Tiling} below for a graphical illustration.

In the following, we denote by $\delta A$ the orthogonal projection of a region $A\subset\overline{\mathbb{H}}$ onto $\mathbb{R}$. Now consider the quantity in Theorem~\ref{th:main_result}, namely
\begin{equation*}
    \mathbb{E}\left[\frac{\mu^{\mathrm{H}}(Q_r)^{p}}{\mu^{\partial}(\delta Q_r)}\right].
\end{equation*}
We apply the localization trick at the boundary, which consists in rewriting the above display as
\begin{equation*}
    \mathbb{E}\left[\frac{\mu^{\partial}(\delta Q_r)}{\mu^{\partial}(\delta Q_r)}\cdot \frac{\mu^{\mathrm{H}}(Q_r)^{p}}{\mu^{\partial}(\delta Q_r)}\right]=\int_{-r}^{r}\mathbb{E}\left[\frac{\mu^{\mathrm{H}}_v(Q_r)^{p}}{\mu^{\partial}_v(\delta Q_r)^2}\right]dv
\end{equation*}
where we used the Girsanov transformation for the last equality on the term $\mu^{\partial}(\delta Q_r)$, and the measures $\mu^{\mathrm{H}}_v$ and $\mu^{\partial}_{v}$ are the bulk and boundary Gaussian multiplicative chaos measures with boundary localizations (and under the exact scaling kernel) defined as
\begin{equation}\label{eq:Localized measures}
\begin{split}
    \mu^{\mathrm{H}}_v(A)&=\int_{A}|z-v|^{-\gamma^2}d\mu^{\mathrm{H}}(z);\\
    \mu^{\partial}_v(I)&=\int_{I}|w-v|^{-\frac{\gamma^2}{2}}d\mu^{\partial}(w).
\end{split}
\end{equation}
where $v\in\mathbb{R}$ is a point located at the boundary and $A\subset\overline{\mathbb{H}}$, $I\subset\mathbb{R}$. See~\cite{Rhodes_2019} for the original localization trick: here we choose to perform the Girsanov transformation to the boundary measure $\mu^{\partial}$ by our discussion in the previous subsection.

It turns out that there is no much difference between different $v\in[-r,r]$, so that we continue our heuristics with $v=0$. Then there is an exact scaling relation for the integrand at least in the exact-scaling case (see Lemma~\ref{lemm:exact_scaling_relation_for_quotients}):
\begin{equation*}
    \mathbb{E}\left[\frac{\mu^{\mathrm{H}}_0(Q_{2^{-n}r})^{p}}{\mu^{\partial}_0(\delta Q_{2^{-n}r})^2}\right]=2^{-n\widetilde{\zeta}(p;2)}\mathbb{E}\left[\frac{\mu^{\mathrm{H}}_0(Q_{r})^{p}}{\mu^{\partial}_0(\delta Q_{r})^2}\right],
\end{equation*}
where $\widetilde{\zeta}(p;2)=\left(2-\frac{\gamma^2}{2}\right)(p-1)-\gamma^2(p-1)^2$. Notice that, as $n$ goes to infinity, the factor $2^{-n\widetilde{\zeta}(p;2)}$ goes to $0$ when $1<p<\frac{2}{\gamma^2}+\frac{1}{2}$, and explodes when $p>\frac{2}{\gamma^2}+\frac{1}{2}$. This explains the threshold $\frac{2}{\gamma^2}+\frac{1}{2}$ appearing in Theorem~\ref{th:main_result}. The other threshold $\frac{4}{\gamma^2}$ is the classical threshold for the two-dimensional Gaussian multiplicative chaos, see Lemma~\ref{lemm:finiteness_of_quotients_with_positive_distance}.

However for this argument to work, we also have to study the finiteness of the term $\mathbb{E}\left[\frac{\mu^{\mathrm{H}}_0(Q_{r})^{p}}{\mu^{\partial}_0(\delta Q_{r})^2}\right]$ in the right-hand side of the above display. This is the localized expression for the $(p;2)$-joint moment. Let's suppose that $p>1$ and continue with Minkowski's inequality, from which we have
\begin{equation*}
    \mathbb{E}\left[\frac{\mu^{\mathrm{H}}_0(Q_r)^{p}}{\mu^{\partial}_0(\delta Q_r)^2}\right]^{1/p}\leq \sum_{n\geq 1}\mathbb{E}\left[\frac{\mu^{\mathrm{H}}_0(\Pi_{2^{-n}r})^{p}}{\mu^{\partial}_0(\delta Q_r)^2}\right]^{1/p}\leq \sum_{n\geq 1}\mathbb{E}\left[\frac{\mu^{\mathrm{H}}_0(\Pi_{2^{-n}r})^{p}}{\mu^{\partial}_0(\delta Q_{2^{-n+1}r})^2}\right]^{1/p}.
\end{equation*}
In the last sum, the summands are related by the localized scaling relations (see Lemma~\ref{lemm:exact_scaling_relation_for_localized_quotients}), namely
\begin{equation*}
    \mathbb{E}\left[\frac{\mu^{\mathrm{H}}_0(\Pi_{2^{-n}r})^{p}}{\mu^{\partial}_0(\delta Q_{2^{-n+1}r})^2}\right]=2^{-n\widetilde{\zeta}(p;2)}\mathbb{E}\left[\frac{\mu^{\mathrm{H}}_0(\Pi_{2^{-1}r})^{p}}{\mu^{\partial}_0(\delta Q_{r})^2}\right].
\end{equation*}
Therefore, in the region $1<p<\frac{2}{\gamma^2}+\frac{1}{2}$, the same argument before shows that the sum over $n$ above converges if the first term in the sum $\mathbb{E}\left[\frac{\mu^{\mathrm{H}}_0(\Pi_{2^{-1}r})^{p}}{\mu^{\partial}_0(\delta Q_{r})^2}\right]$ is finite. Notice that the region $\Pi_{2^{-1}r}$ in the bulk measure of the nominator of this expectation is of some strictly positive distance from the localized insertion point at $v=0$, and it is not difficult then to see that this expectation term is upper bounded by the (unlocalized) $(p;2)$-joint moment of the bulk/boundary quotient.

To summarize this heuristic discussion, if we can first establish criteria for the finiteness of the unlocalized $(p;2)$-joint moment of the bulk/boundary quotient, then we can say something rather precise about the finiteness of the unlocalized $(p;1)$-joint moment of the bulk/boundary quotient. More generally, similar considerations show that we can reduce the finiteness of the unlocalized $(p;q)$-joint moment of the bulk/boundary quotient to that of the unlocalized $(p;q+1)$-joint moment. This is the induction scheme made possible by the localization at the boundary trick, and to finalize our proof it suffices to show some weak a priori finiteness condition of the $(p;q)$-joint moment of the bulk/boundary quotient for large $q$. This is done essentially by observing that, for large $q$, the threshold is simple as $\min(\frac{2}{\gamma^2}+\frac{q}{2},\frac{4}{\gamma^2})=\frac{4}{\gamma^2}$, but this is the threshold given by a classical two-dimensional Gaussian multiplicative chaos, which corresponds to the part in the bulk measure $\mu^{\mathrm{H}}(Q)$ with positive distance away from the boundary $\mathbb{R}$. This provides the base step for our induction scheme (from above).

To make this discussion into a easy-to-remember version: near the boundary, the $q$-th power of the boundary measure in the denominator ``cancels'' $\frac{q}{2}$-th power of the bulk measure in the nominator in the expression of the bulk/boundary quotient as we get close to the critical joint moment bound. Therefore, the moment bound of the part near the boundary is reduced to the main result of~\cite{Huang:2023aa} recalled in~\eqref{eq:Huang23Result}, which translates to the condition $p-\frac{q}{2}<\frac{2}{\gamma^2}$ with this cancellation. Away from the boundary, the bulk measure is a classical two-dimensional Gaussian multiplicative chaos, and has moment bound $p<\frac{4}{\gamma^2}$.

We now rigorously implement this argument in the following section.

\section{Joint moment bounds and proofs of the main result}\label{sec:preliminary_estimates_on_the_bulk_boundary_quotient}
As a preparation to the subsequent papers in this series, our original motivation was to show the following quantity, which is a quotient of bulk/boundary measures,
\begin{equation*}
    \mathbb{E}\left[\frac{\mu^{\mathrm{H}}(Q)^{\frac{2}{\gamma^2}}}{\mu^{\partial}(I)}\right]
\end{equation*}
is finite, where $I$ is the orthogonal projection of a Carleson cube $Q$ onto $\mathbb{R}$. This corresponds to the non-localized version of the quantity appearing in the right tail profile constant of bulk Gaussian multiplicative chaos measure $\mu(Q)$, and we need this kind of preliminary bounds to establish the expression of the above-mentioned tail profile constant in subsequent papers. We will give precise sufficient conditions for the finiteness of more general expressions of type
\begin{equation*}
    \mathbb{E}\left[\frac{\mu^{\mathrm{H}}(Q)^{p}}{\mu^{\partial}(I)^{q}}\right],
\end{equation*}
see Theorems~\ref{th:main_result_positive}.

Before we start, let us remark at once that if we can prove the main Theorem~\ref{th:main_result_positive} for the exact scaling field $X_\mathrm{C}$ with covariance~\eqref{eq:ExactScalingKernel}, then Lemma~\ref{lemm:FinitenessBulkBoundaryGeneralKernel} immediate yields the same conclusions for general log-correlated fields with covariance~\eqref{eq:GeneralCovarianceKernelHalfPlane}. In the rest of this section, $X$ will always be in the exact scale-invariant form, and we drop the subscript $\mathrm{C}$ for simplicity.

\subsection{Finiteness of quotients with positive distance}
When $A$ is a region of $\overline{\mathbb{H}}$, we write $\delta A$ its orthogonal projection to $\mathbb{R}$. Later, the notation $\partial A$ will be used to denote the intersection of $A$ with $\mathbb{R}$. For example, when $Q$ is a Carleson cube, both $\delta Q$ and $\partial Q$ denote its intersection with $\mathbb{R}$.

To explain the appearance of the factor $\frac{4}{\gamma^2}$ in all the main results, we record the following preliminary observation.
\begin{lemm}[Finiteness of quotients with positive distance]\label{lemm:finiteness_of_quotients_with_positive_distance}
If a bounded region $A$ of the upper-half plane $\mathbb{H}$ is of some positive distance to $\mathbb{R}$, then
\begin{equation*}
    \mathbb{E}\left[\frac{\mu^{\mathrm{H}}(A)^{p}}{\mu^{\partial}(\delta A)^{q}}\right]<\infty
\end{equation*}
if and only if $p<\frac{4}{\gamma^2}$ and $-q<\frac{4}{\gamma^2}$.
\end{lemm}
This also immediate implies that even if $A$ is not of positive distance to $\mathbb{R}$, $\mathbb{E}\left[\frac{\mu^{\mathrm{H}}(A)^{p}}{\mu^{\partial}(\delta A)^{q}}\right]$ blows up when $p\geq\frac{4}{\gamma^2}$ by restricting $A$ to a small set of some positive distance to $\mathbb{R}$.
\begin{proof}
We can use the generalized Kahane's inequality of Theorem~\ref{th:GeneralKahane} to decorrelate the log-correlated field $X_\mathrm{N}$ restricted to $A$ and to $\delta A$, and this results in
\begin{equation*}
    \frac{1}{C}\mathbb{E}[\mu(A)^{p}]\mathbb{E}\left[\mu^{\partial}(\delta A)^{-q}\right]\leq \mathbb{E}\left[\frac{\mu(A)^{p}}{\mu^{\partial}(\delta A)}\right]\leq C\mathbb{E}[\mu(A)^{p}]\mathbb{E}\left[\mu^{\partial}(\delta A)^{-q}\right]
\end{equation*}
for some finite positive constant $C$. The second inequality is due to Lemma~\ref{lemm:DecorrelateBulkBoundaryPositiveDistance}, and the first inequality follows from similar considerations. More precisely, the positive distance assumption implies that there exists some $K>0$ such that for all $x\in A$ and $y\in I$, we have $\mathbb{E}[X(x)X(y)]\leq K$. Consider independent copies $X_A$ and $X_I$ respectively on $A$ and $I$ of the field $X$, and modify them to $\widetilde{X_A}=X_A+\sqrt{K}N$ and $\widetilde{X_I}=X_I+\sqrt{K}N$ with the same but independent Gaussian $N$. For the field $X$, we also modify it in the following way: let $N_A,N_I$ be other independent Gaussians, if $x\in A$, we define $\widehat{X}(x)=X(x)+\sqrt{K}N_A$, and if $x\in I$, $\widehat{X}(y)=X(y)+\sqrt{K}N_I$. We have this time $\mathbb{E}[\widehat{X}(x)\widehat{X}(y)]\leq\mathbb{E}[\widetilde{X_A}(x)\widetilde{X_I}(y)]$ if $(x,y)\in (A\times I)\cup (I\times A)$ and equal otherwise. Therefore, we get the reversed inequality via the generalized Kahane's inequality of Theorem~\ref{th:GeneralKahane}, following the same lines as in the proof of Lemma~\ref{lemm:DecorrelateBulkBoundaryPositiveDistance}. The lemma follows from classical moment bounds on Gaussian multiplicative chaos measures (for the boundary measure, the parameter is $\frac{\gamma}{2}$ but the covariance is $-2\ln$, resulting in the moment bound $\frac{4}{\gamma^2}$ instead of $\frac{2}{\gamma^2}$).
\end{proof}

However, investigating the case where the cube $Q$ is adjacent to $\mathbb{R}$ is our main goal of this section. We need some preliminary calculations on the exact scaling relations on the bulk/boundary quotients of Gaussian multiplicative chaos measures, with or without extra pointwise singularity at the boundary.

\subsection{Exact scaling relations for bulk/boundary quotients}\label{subse:exact_scaling_relations}
The following exact scaling relations are already used in~\cite{huang2018liouville,Huang:2023aa}.
\begin{lemm}[Exact scaling relation for the bulk measure]
Denote by
\begin{equation*}
    \overline{\zeta}(p;0)=\left(2+\frac{\gamma^2}{2}\right)p-\gamma^2p^2.
\end{equation*}
Then for any $A\subset Q_r$ and $p\in\mathbb{R}$,
\begin{equation*}
    \mathbb{E}[\mu^{\mathrm{H}}(rA)^{p}]=r^{\overline{\zeta}(p)}\mathbb{E}[\mu^{\mathrm{H}}(A)^{p}].
\end{equation*}
\end{lemm}

This lemma (as well as other lemmas in this section) includes the case where both expectations in the above display are infinite, and the index $(p;0)$ indicates that we are looking at the $(p,0)$-joint moment of the bulk/boundary quotient. We refer to~\cite[Lemma~4]{Huang:2023aa} for a proof of this relation (or one can consult the following proof, which is a bit more complicated). We need some slight generalizations for the purpose of this article:

\begin{lemm}[Exact scaling relation for bulk/boundary quotients]\label{lemm:exact_scaling_relation_for_quotients}
For any $A\subset Q_r$ and $p,q\in\mathbb{R}$,
\begin{equation*}
    \mathbb{E}\left[\frac{\mu^{\mathrm{H}}(rA)^{p}}{\mu^{\partial}(\delta(rA))^{q}}\right]=r^{\overline{\zeta}(p-\frac{q}{2};0)}\mathbb{E}\left[\frac{\mu^{\mathrm{H}}(A)^{p}}{\mu^{\partial}(\delta A)^{q}}\right].
\end{equation*}
\end{lemm}
This lemma also includes the case where both expectations are infinite. We will denote thus by
\begin{equation}\label{eq:quotient_scaling_exponent}
    \overline{\zeta}(p;q)=\overline{\zeta}(p-\frac{q}{2};0)=\left(2+\frac{\gamma^2}{2}\right)(p-\frac{q}{2})-\gamma^2(p-\frac{q}{2})^2
\end{equation}
the scaling exponent function for general $(p,q)$-joint moments of the bulk/boundary quotient.

\begin{proof}
The proof is a straightforward adaptation of the proof of the previous lemma. We include a brief proof since this is not written elsewhere, and refer to~\cite[Lemma~4]{Huang:2023aa} for the details with the $\epsilon$-mollifier and the infinite expectation case.

Recall the following equality in law, consequence of the exact scaling kernel~\eqref{eq:ExactScalingKernel}:
\begin{equation*}
    \{X_\epsilon(z)+\sqrt{-2\ln r}N\}_{z\in Q_r}=\{X_{r\epsilon}(rz)\}_{z\in Q_r}
\end{equation*}
where $N$ is an independent standard Gaussian variable. Applying this to the lemma, we have the following equality in law:
\begin{equation*}
\begin{split}
    \frac{\mu^{\mathrm{H}}(rA)^{p}}{\mu^{\partial}(\delta(rA))^{q}}&=\lim_{\epsilon\to 0}\frac{\left(\int_{rA}e^{\gamma X_{r\epsilon}(z)-\frac{\gamma^2}{2}\mathbb{E}[X_{r\epsilon}(z)^2]}\text{Im}(z)^{-\frac{\gamma^2}{2}}d^2z\right)^{p}}{\left(\int_{\delta (rA)}e^{\frac{\gamma}{2}X_{r\epsilon}(w)-\frac{\gamma^2}{8}\mathbb{E}[X_{r\epsilon}(w)^2]}dw\right)^{q}}\\
    &=\lim_{\epsilon\to 0}\frac{(e^{\gamma\sqrt{-2\ln r}N-\frac{\gamma^2}{2}(-2\ln r)}r^{-\frac{\gamma^2}{2}}r^2)^{p}\cdot\left(\int_{A}e^{\gamma X_{\epsilon}(z)-\frac{\gamma^2}{2}\mathbb{E}[X_\epsilon(z)^2]}\text{Im}(z)^{-\frac{\gamma^2}{2}}d^2z\right)^{p}}{(e^{\frac{\gamma}{2}\sqrt{-2\ln r}N-\frac{\gamma^2}{8}(-2\ln r)}r)^{q}\cdot\left(\int_{\delta A}e^{\frac{\gamma}{2}X_\epsilon(w)-\frac{\gamma^2}{8}\mathbb{E}[X_\epsilon(w)^2]}dw\right)^{q}}\\
    &=e^{\gamma\sqrt{-2\ln r}(p-\frac{q}{2})}\cdot r^{(2+\frac{\gamma^2}{2})(p-\frac{q}{2})}\cdot\frac{\mu^{\mathrm{H}}(A)^{p}}{\mu^{\partial}(\delta A)^{q}}.
\end{split}
\end{equation*}
The claim follows from taking the expectation
\begin{equation*}
    \mathbb{E}\left[e^{\gamma\sqrt{-2\ln r}(p-\frac{q}{2})}\right]=r^{-\gamma^2(p-\frac{q}{2})^2},
\end{equation*}
and the fact that the random variable $N$ is independent of the rest.
\end{proof}

\begin{rema}
Notice that the choice of $\text{Im}(z)^{-\frac{\gamma^2}{2}}$ in the background singularity for the bulk measure~\eqref{eq:DefinitionBulkMeasure} is important to have this nice shift relation by $-\frac{q}{2}$. This special choice (which appears naturally in boundary Liouville conformal field theory) makes the scaling factor in the bulk region exactly a multiple of the scaling factor on the boundary. Indeed, our method is general and any other singularity strength can be similarly treated, but the results would not be as nice and simple as the ones in our main theorems. The shift $-\frac{q}{2}$ in the scaling exponent function corresponds to our heuristics that ``the boundary measure behaves like the square root of the bulk measure'', and is responsible for the shift $+\frac{q}{2}$ in the moment bounds in Theorem~\ref{th:main_result_positive}.
\end{rema}

\subsection{Exact scaling relations for bulk/boundary quotients with boundary singularity}\label{subse:exact_scaling_relations_for_bulk_boundary_quotients_with_localization_at_the_boundary}
We also need the following generalizations, where a pointwise singularity is added to the boundary for both the bulk and the boundary Gaussian multiplicative chaos measures, see~\eqref{eq:Localized measures}. We have similar scaling relations for the quotient of these measures:
\begin{lemm}[Exact scaling relation for localized bulk/boundary quotients]\label{lemm:exact_scaling_relation_for_localized_quotients}
For simplicity, we suppose that the measures are localized at $v=0$ (the general case is similar by translation along $\mathbb{R}$). Then for any $A\subset Q_r$ and $p\in\mathbb{R}$,
\begin{equation*}
    \mathbb{E}\left[\frac{\mu^{\mathrm{H}}_{v=0}(rA)^{p}}{\mu_{v=0}^{\partial}(\delta(rA))^{q}}\right]=r^{\widetilde{\zeta}(p;q)}\mathbb{E}\left[\frac{\mu^{\mathrm{H}}_{v=0}(A)^{p}}{\mu_{v=0}^{\partial}(\delta A)^{q}}\right],
\end{equation*}
where we define the localized boundary scaling exponent by
\begin{equation}\label{eq:localized_scaling_exponent}
    \widetilde{\zeta}(p;q)=\overline{\zeta}(p-\frac{q}{2};0)-\gamma^2(p-\frac{q}{2})=(2-\frac{\gamma^2}{2})(p-\frac{q}{2})-\gamma^2(p-\frac{q}{2})^2.
\end{equation}
\end{lemm}
\begin{proof}
The proof is similar to the one above and details are omitted. By an exact scaling calculation as the previous one we get
\begin{equation*}
    \mathbb{E}\left[\frac{\mu^{\mathrm{H}}_{v=0}(rA)^{p}}{\mu_{v=0}^{\partial}(\delta(rA))^{q}}\right]=e^{\gamma\sqrt{-2\ln r}(p-\frac{q}{2})}\cdot r^{(2-\frac{\gamma^2}{2})(p-\frac{q}{2})}\cdot \mathbb{E}\left[\frac{\mu^{\mathrm{H}}_{v=0}(A)^{p}}{\mu_{v=0}^{\partial}(\delta A)^{q}}\right].
\end{equation*}
and the claim follows from the independence of the Gaussian factor $N$.
\end{proof}

\begin{rema}[Sign-change of the scaling exponent functions]\label{rema:sign-change}
Notice that the unlocalized/localized scaling exponent functions $\overline{\zeta}(p;q)$ and $\widetilde{\zeta}(p;q)$ are both trivial when $q=2p$. It is also useful to observe the sign-change of $\widetilde{\zeta}(p;q)$:
\begin{enumerate}
    \item When $2p>q$, the function $\widetilde{\zeta}(p;q)$ is positive if and only if $p-\frac{q}{2}<\frac{2}{\gamma^2}-\frac{1}{2}$;
    \item When $2p<q$, the function $\widetilde{\zeta}(p;q)$ is positive if and only if $p-\frac{q}{2}>\frac{2}{\gamma^2}-\frac{1}{2}$.
\end{enumerate}

We will also need the following sign-change of the unlocalized scaling exponent function $\overline{\zeta}(p;q)-1$: this quantity is positive if and only if $p-\frac{q}{2}\in(\frac{1}{2},\frac{2}{\gamma^2})$, with $\gamma\in(0,2)$.
\end{rema}

\subsection{Finiteness of the joint moment: case with $q\in(2p-\frac{4}{\gamma^2},2p-1)$}
We now establish Theorems~\ref{th:main_result_positive} in a special case where $q$ is close to $2p-1$ from below. Namely, we suppose in this subsection that
\begin{equation}\label{eq:large_q_condition}
    2p-\frac{4}{\gamma^2}<q<2p-1.
\end{equation}
Notice that the above condition is non-empty since $\gamma\in(0,2)$, and in this case we always have $p<\frac{2}{\gamma^2}+\frac{q}{2}$, and the threshold appearing in Theorem~\ref{th:main_result_positive} is reduced to $p<\min(\frac{2}{\gamma^2}+\frac{q}{2},\frac{4}{\gamma^2})=\frac{4}{\gamma^2}$, which is the classical moment bound for the two-dimensional Gaussian multiplicative chaos.

Our goal is to establish a sufficient condition for the finiteness of $(p;q)$-joint moment of the bulk/boundary quotient of Gaussian multiplicative chaos measures under the supplementary condition of this section, using the Whitney decomposition picture of~\cite{huang2018liouville,Huang:2023aa}. Without loss of generality, consider the Carleson cube $Q=Q_r=[-r,r]\times[0,2r]$ for some small $r>0$. Define the ``upper, left, right'' parts of $Q_r$ by
\begin{equation*}
    Q^{\mathrm{U}}_r=[-r,r]\times[r,2r],\quad Q^{\mathrm{L}}_r=[-r,0]\times[0,r],\quad Q^{\mathrm{R}}_r=[0,r]\times[0,r].
\end{equation*}
We can reiterate this decomposition on the Carleson cubes $Q^{\mathrm{L}}$ and $Q^{\mathrm{R}}$, and this procedure can be repeated infinitely many times. In the end, we tile the original Carleson cube $Q_r$ by infinitely many smaller cubes of the same shape as $Q^{\mathrm{U}}$, but rescaled: for each $n\geq 0$, there are $2^{n}$ cubes of height $2^{-n}r$ in the rectanglar region
\begin{equation*}
    L_{r,n}=[-r,r]\times [2^{-n}r, 2^{-n-1}r],
\end{equation*}
and we denotes these cubes by $(Q^{(n)}_{i})_{i=1,\dots,2^{n}}$ (we will omit the subscript $r$ for simplicity). See Figure~\ref{fig:InfiniteWhitney} for a picture of this decomposition.

\begin{figure}[h]
\centering
\includegraphics[height=15em]{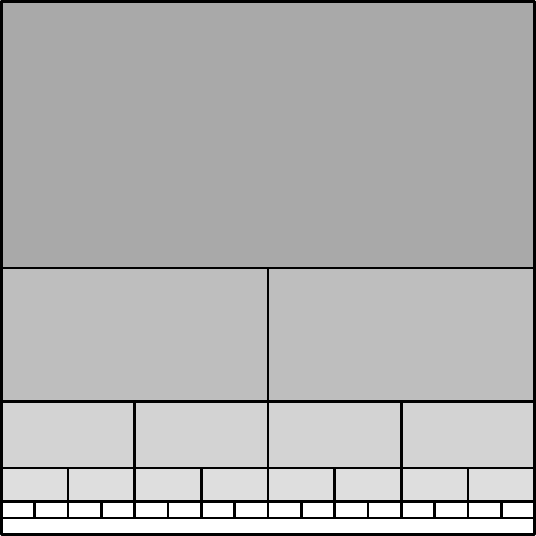}
\caption{Dividing a Carleson cube into infinitely many smaller cubes $(Q^{(n)}_i)_{n\geq 0, i=1\dots,2^{n}}$.}
\label{fig:InfiniteWhitney}
\end{figure}

By the exact scaling relation for the unlocalized bulk/boundary quotients of Lemma~\ref{lemm:exact_scaling_relation_for_quotients}, we have
\begin{equation*}
    \forall n\geq 0, \forall 1\leq i\leq 2^{n},\quad \mathbb{E}\left[\frac{\mu^{\mathrm{H}}(Q^{(n)}_{i})^{p}}{\mu^{\partial}(\delta Q^{(n)}_{i})^{q}}\right]=2^{-n\overline{\zeta}(p;q)}\mathbb{E}\left[\frac{\mu^{\mathrm{H}}(Q^{\mathrm{U}})^{p}}{\mu^{\partial}(\delta Q^{\mathrm{U}})^{q}}\right].
\end{equation*}
and the last expectation is finite under the condition that $p<\frac{4}{\gamma^2}$ by Lemma~\ref{lemm:finiteness_of_quotients_with_positive_distance}.

On the other hand,
\begin{itemize}
    \item If $0<p<1$, then by the subadditivity inequality,
    \begin{equation*}
    \begin{split}
        \mathbb{E}\left[\frac{\mu^{\mathrm{H}}(Q)^{p}}{\mu^{\partial}(\delta Q)^{q}}\right]&\leq\sum_{n=0}^{\infty}\sum_{i=1}^{2^{n}}\mathbb{E}\left[\frac{\mu^{\mathrm{H}}(Q^{(n)}_{i})^{p}}{\mu^{\partial}(\delta Q)^{q}}\right]\\
        &\leq \sum_{n=0}^{\infty}\sum_{i=1}^{2^{n}}\mathbb{E}\left[\frac{\mu^{\mathrm{H}}(Q^{(n)}_{i})^{p}}{\mu^{\partial}(\delta Q^{(n)}_{i})^{q}}\right]\\
        &=\sum_{n=0}^{\infty}2^{n(1-\overline{\zeta}(p;q))}\mathbb{E}\left[\frac{\mu^{\mathrm{H}}(Q^{\mathrm{U}})^{p}}{\mu^{\partial}(\delta Q)^{q}}\right].
    \end{split}
    \end{equation*}
    \item If $p\geq 1$, then by Minkowski's inequality for the $L_p$-norm,
    \begin{equation*}
    \begin{split}
        \mathbb{E}\left[\frac{\mu^{\mathrm{H}}(Q)^{p}}{\mu^{\partial}(\delta Q)^{q}}\right]^{\frac{1}{p}}&\leq\sum_{n=0}^{\infty}\sum_{i=1}^{2^{n}}\mathbb{E}\left[\frac{\mu^{\mathrm{H}}(Q^{(n)}_{i})^{p}}{\mu^{\partial}(\delta Q)^{q}}\right]^{\frac{1}{p}}\\
        &\leq \sum_{n=0}^{\infty}\sum_{i=1}^{2^{n}}\mathbb{E}\left[\frac{\mu^{\mathrm{H}}(Q^{(n)}_{i})^{p}}{\mu^{\partial}(\delta Q^{(n)}_{i})^{q}}\right]^{\frac{1}{p}}\\
        &=\sum_{n=0}^{\infty}2^{\frac{n}{p}(1-\overline{\zeta}(p;q))}\mathbb{E}\left[\frac{\mu^{\mathrm{H}}(Q^{\mathrm{U}})^{p}}{\mu^{\partial}(\delta Q)^{q}}\right]^{\frac{1}{p}}.
    \end{split}
    \end{equation*}
\end{itemize}

In all cases, the sum is exponentially converging in $n$ for $q\in(2p-\frac{4}{\gamma^2},2p-1)$ with $\gamma\in(0,2)$ by Remark~\ref{rema:sign-change}. This finishes the proof of the finiteness of the $(p;q)$-joint moment of the bulk/boundary quotient under the condition that $q\in(2p-\frac{4}{\gamma^2},2p-1)$ (and $p,q>0$).

\begin{rema}[Explosion of the joint bulk/boundary moment at the critical threshold for large $q$]\label{rema:explosion_at_criticality_large_q}
Notice that we have also proven the explosion at the critical threshold $p_c=\frac{4}{\gamma^2}$ in this case. Indeed, first observe that
\begin{equation*}
    \mathbb{E}\left[\frac{\mu^{\mathrm{H}}(Q)^{p}}{\mu^{\partial}(I)^{q}}\right]\geq \mathbb{E}\left[\frac{\mu^{\mathrm{H}}(Q^{\mathrm{U}})^{p}}{\mu^{\partial}(I)^{q}}\right]\geq C\mathbb{E}\left[\mu^{\mathrm{H}}(Q^{\mathrm{U}})^{p}\right]\mathbb{E}\left[\mu^{\partial}(I)^{-q}\right]
\end{equation*}
by Lemma~\ref{lemm:finiteness_of_quotients_with_positive_distance}. But $\mathbb{E}\left[\mu^{\mathrm{H}}(Q^{\mathrm{U}})^{p}\right]$ explodes when $p\geq \frac{4}{\gamma^2}$ (and the negative moment of the boundary measure is always non-trivial~\cite[Theorem~2.12]{Rhodes_2014}).
\end{rema}

\subsection{Finiteness of the joint moment: case with large enough $q$}\label{sub:finiteness_of_the_joint_moment_case_with_large_enough_q_}
We now show Theorem~\ref{th:main_result_positive} in the case where $q\geq 2p-1$. Notice that the argument in the above subsection does not work directly in the case $q\in[2p-1,\infty)$. Especially, when $q=2p$, the scaling exponents $\overline{\zeta}(p;q)$ are trivially $0$, and no exponential convergence can be obtained directly by the previously described method. Nonetheless, a small modification suffices to get Theorem~\ref{th:main_result_positive} for all $q\geq 2p-1$. The key point is to remember that the boundary mass $\mu^{\partial}(I)$ has non-trivial negative moments of any order~\cite[Theorem~2.12]{Rhodes_2014}.

The fact that the joint bulk/boundary moments explode when $p\geq \frac{4}{\gamma^2}$ for any $q\geq 2p-1$ follows from Remark~\ref{rema:explosion_at_criticality_large_q} and exactly the same argument as in the previous section. For the other direction, suppose that $p<\frac{4}{\gamma^2}$ and $q\geq 2p-1$. Then for any conjugate pair $\frac{1}{a}+\frac{1}{b}=1$ and $a,b>1$, by H\"older's inequality, for any $\eta>0$,
\begin{equation*}
    \mathbb{E}\left[\frac{\mu^{\mathrm{H}}(Q)^{p}}{\mu^{\partial}(I)^{q}}\right]\leq \mathbb{E}\left[\frac{\mu^{\mathrm{H}}(Q)^{ap}}{\mu^{\partial}(I)^{a(q-\eta)}}\right]^{\frac{1}{a}}\mathbb{E}\left[\mu^{\partial}(I)^{-b\eta}\right]^{\frac{1}{b}}.
\end{equation*}
Since $p<\frac{4}{\gamma^2}$ and $q\geq 2p-1$, we can choose a pair of parameters $(a,\eta)$, with $a>1$ such that $ap<\frac{4}{\gamma^2}$ and $\eta>0$ such that $a(q-\eta)\in(2ap-\frac{4}{\gamma^2},2ap-1)$. Then by the result of the previous subsection,
\begin{equation*}
    \mathbb{E}\left[\frac{\mu^{\mathrm{H}}(Q)^{ap}}{\mu^{\partial}(I)^{a(q-\eta)}}\right]^{\frac{1}{a}}<\infty.
\end{equation*}
But we also recalled in the beginning of this subsection that $\mathbb{E}\left[\mu^{\partial}(I)^{-b\eta}\right]<\infty$. Therefore,
\begin{equation*}
    \mathbb{E}\left[\frac{\mu^{\mathrm{H}}(Q)^{p}}{\mu^{\partial}(I)^{q}}\right]<\infty
\end{equation*}
as long as $p<\frac{4}{\gamma^2}$ and $q\geq 2p-1$, which is our main claim of this subsection.

To sum up, with the results in these two subsections, we have shown Theorem~\ref{th:main_result_positive} in the case where $q>2p-\frac{4}{\gamma^2}$. Furthermore, we have shown in this case the explosion of the joint bulk/boundary moment at the critical threshold (which is reduced to $\frac{4}{\gamma^2}$ for this range of $q$).

\begin{rema}
One might wonder we take a snake path towards proving Theorem~\ref{th:main_result_positive} in the case for $q\in[2p-1,2p]$ (especially say with $q=2p$). This can be explained by analogy to~\cite{Huang:2023aa} in the $q=0$ case. Indeed, due to the shift $-\frac{q}{2}$ induced by the boundary measure denominator, the bulk/boundary quotient term with $q\in[2p-1,2p]$ behaves similarly to the $p'$-th moment of the bulk measure for $p'\in[0,\frac{1}{2})$. One can inspect the proof of~\cite[Lemma~8]{Huang:2023aa} and observe that, the proof is first done in the case $p'\in(\frac{1}{2},\frac{2}{\gamma^2})$ and the argument of this part is not possible if $p'\in(0,\frac{1}{2})$ due to the sign-change in the scaling exponent, and we only get the case $p'\in(0,\frac{1}{2})$ afterwards by H\"older's inequality and the trivialness of the $0$-th moment of the bulk measure. However, the case of the $0$-th moment for the bulk measure problem corresponds to the problematic critical case $p=2q$ for the joint $(p,q)$-moment, and our solution in this subsection is to first go to the $(0,q')$-joint moment for some higher values of $q'$, which rougly corresponds to using some negative moment (instead of the $0$-th moment) of the bulk measure in the context of~\cite{Huang:2023aa}.
\end{rema}

\subsection{Proof of Theorem~\ref{th:main_result_positive}}
As a first step, let us explain what is the localization trick at the boundary via the Girsanov transformation. Recall that our goal was to study expressions of type
\begin{equation*}
    \mathbb{E}\left[\frac{\mu^{\mathrm{H}}(Q)^{p}}{\mu^{\partial}(I)^{q}}\right],
\end{equation*}
for $Q=Q_r=[-r,r]\times[0,2r]$ and $I=[-r,r]$ (we will drop the $r$ subscript in the following to leave place for other indices). We can rewrite this as
\begin{equation}\label{eq:deadlyintegral}
    \mathbb{E}\left[\mu^{\partial}(I)\cdot\frac{\mu^{\mathrm{H}}(Q)^{p}}{\mu^{\partial}(I)^{q+1}}\right]=\int_{-r}^{r}\mathbb{E}\left[\frac{\mu_v^{\mathrm{H}}(Q)^{p}}{\mu_v^{\partial}(I)^{q+1}}\right]dv,
\end{equation}
where we applied the Girsanov transformation to the factor $\mu^{\partial}(I)$ to get the localized measures~\eqref{eq:Localized measures}, similar to~\cite{Rhodes_2019} where it was applied to $\mu^{\mathrm{H}}(Q)$ instead. Therefore, to show explosion in~\eqref{eq:deadlyintegral}, it suffices to show that the integrand explodes for all $v$ in some small interval $[-\epsilon,\epsilon]$. To show the finiteness of~\eqref{eq:deadlyintegral}, we can show that the integrand is finite for each $v\in[-r,r]$ and that it would not grow too quickly as $v$ approaches the end points $\{-r,r\}$.

We now prove Theorem~\ref{th:main_result_positive} on the finiteness of the joint moments of bulk/boundary quotient. The explosion of the joint moment in the case $p\geq \frac{4}{\gamma^2}$ is already treated previously as in Remark~\ref{rema:explosion_at_criticality_large_q}. Therefore, suppose that $p<\frac{4}{\gamma^2}$ and we study the finiteness of the integral~\eqref{eq:deadlyintegral}. The following remark is a general observation recorded in~\cite{Rhodes_2019}.
\begin{rema}
By symmetry, suppose that $v\in[0,r]$ and denote by $\rho=r-v$ the distance from $v$ to the nearest end point $r$. To show Theorem~\ref{th:main_result_positive}, it suffices to prove that, under the corresponding conditions of Theorem~\ref{th:main_result_positive},
\begin{equation}\label{eq:C_1_C_2_delta}
    \mathbb{E}\left[\frac{\mu_v^{\mathrm{H}}(Q)^{p}}{\mu_v^{\partial}(I)^{q+1}}\right]\leq C_1+C_2 \rho^{-\eta_{(p;q+1)}}
\end{equation}
for some coefficient $\eta_{(p;q+1)}>-1$, and $C_1,C_2$ independent of $\rho$, i.e. independent of the position of $v$. Indeed, plugging this bound in the integral~\eqref{eq:deadlyintegral} yields the finiteness of $\mathbb{E}\left[\frac{\mu^{\mathrm{H}}(Q)^{p}}{\mu^{\partial}(I)^{q}}\right]$.
\end{rema}

Actually, we can do slightly better in our setting in that we can take $C_2=0$ in the above remark.
\begin{lemm}[Finiteness of the localized joint moment from the unlocalized joint moment]\label{lemm:cutecutelemma}
Suppose that the unlocalized $(p,q+1)$-joint moment of the bulk/boundary quotient
\begin{equation}\label{eq:cuteassumption}
    \mathbb{E}\left[\frac{\mu^{\mathrm{H}}(Q)^{p}}{\mu^{\partial}(I)^{q+1}}\right]<\infty
\end{equation}
is finite and $q<2p-1, p<\frac{4}{\gamma^2}, p<\frac{2}{\gamma^2}+\frac{q}{2}$. Then the integrand above, i.e.
\begin{equation}\label{eq:cuteconclusion}
    \mathbb{E}\left[\frac{\mu_v^{\mathrm{H}}(Q)^{p}}{\mu_v^{\partial}(I)^{q+1}}\right]
\end{equation}
is uniformly bounded in $v\in[-r,r]$.
\end{lemm}

We now explain how to finish the proof of Theorem~\ref{th:main_result_positive} given this lemma. Under the condition $p<\min(\frac{2}{\gamma^2}+\frac{q}{2},\frac{4}{\gamma^2})$, Lemma~\ref{lemm:cutecutelemma} and the above discussion show that the $(p,q)$-joint moment of the bulk/boundary quotient is finite if the $(p,q+1)$-joint moment is finite. Since this condition is verified with $q$ changed to $q+1$ as long as $q<2p-1$, we can iterate this step until the first time $q>2p-\frac{4}{\gamma^2}$. Then the conclusions in Section~\ref{sub:finiteness_of_the_joint_moment_case_with_large_enough_q_} holds, so that the induction scheme ends after a finite iteration, and the proof of Theorem~\ref{th:main_result_positive} is complete.

The rest of the article is devoted to proving Lemma~\ref{lemm:cutecutelemma} in a quite elementary fashion. It is instrumental to study some special cases first.

\subsubsection{The end point case}
Suppose that $v=r$ in Lemma~\ref{lemm:cutecutelemma} (the case $v=r$ is similar) and that
\begin{equation*}
    \mathbb{E}\left[\frac{\mu^{\mathrm{H}}(Q)^{p}}{\mu^{\partial}(I)^{q+1}}\right]<\infty.
\end{equation*}

For this, we will cover $Q_r$ using scaled copies of the following $\Gamma$-shaped region defined as the following (see Figure~\ref{fig:Gamma_Tiling}). Let
\begin{equation}\label{eq:Q'_Definition}
    Q'_{\rho}=[r-2\rho,r]\times[0,2\rho]
\end{equation}
be the cubic region around $v=r-\rho$ and
\begin{equation*}
    \Gamma_{\rho}=Q'_{2\rho}\setminus Q'_\rho
\end{equation*}
be the complement of $Q'_\rho$ in $Q'_{2\rho}$, which has the shape of the letter $\Gamma$. In particular, the $\Gamma$-shaped regions
\begin{equation*}
    (\Gamma_{2^{-n}r})_{n\geq 1}
\end{equation*}
is a tiling of the Carleson cube $Q_r$, as depicted in Figure~\ref{fig:Gamma_Tiling}. We also denote by
\begin{equation*}
    \partial\Gamma_{\rho}=[r-4\rho,r-2\rho]
\end{equation*}
the intersection of $\Gamma_{\rho}$ with the boundary $\mathbb{R}$.

\begin{figure}[h]
\centering
\includegraphics[height=15em]{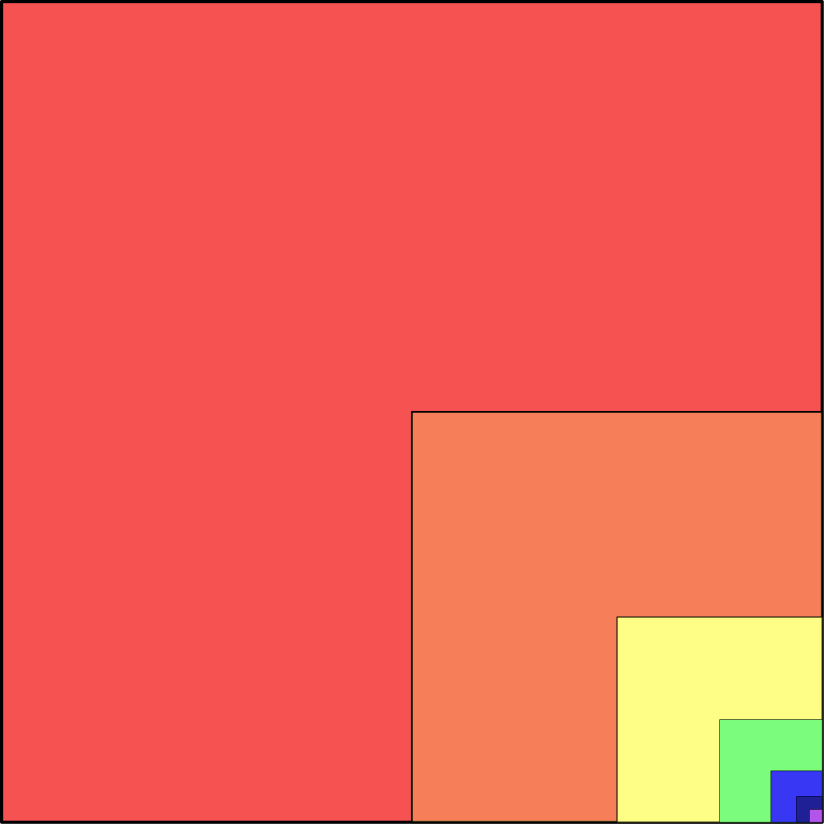}
\caption{Tiling a Carleson cube by $\Gamma$-shaped regions (with respect to the bottom right end point).}
\label{fig:Gamma_Tiling}
\end{figure}

\begin{prop}[Finiteness of joint moments of $\Gamma$-shaped regions]\label{prop:cute_Gamma}
Suppose that $p<\frac{4}{\gamma^2}$ and~\eqref{eq:cuteassumption} holds. Then
\begin{equation}\label{eq:cute_Gamma}
    \mathbb{E}\left[\frac{\mu^{\mathrm{H}}(\Gamma_{2^{-1}r})^{p}}{\mu^{\partial}(\partial \Gamma_{2^{-1}r})^{q+1}}\right]<\infty.
\end{equation}
\end{prop}
\begin{proof}
Notice that $\Gamma_{2^{-1}r}$ is the union of $Q^{\mathrm{U}}$ and $Q^{\mathrm{L}}$, the ``upper'' and ``lower left'' part of the Carleson cube $Q_r$. Suppose that $0\leq p\leq 1$, then by the subadditivity inequality,
\begin{equation*}
    \left[\frac{\mu^{\mathrm{H}}(\Gamma_{2^{-1}r})^{p}}{\mu^{\partial}(\partial \Gamma_{2^{-1}r})^{q+1}}\right]\leq \left[\frac{\mu^{\mathrm{H}}(Q^{\mathrm{L}})^{p}}{\mu^{\partial}(\partial \Gamma_{2^{-1}r})^{q+1}}\right]+\left[\frac{\mu^{\mathrm{H}}(Q^{\mathrm{U}})^{p}}{\mu^{\partial}(\partial \Gamma_{2^{-1}r})^{q+1}}\right].
\end{equation*}
But Lemma~\ref{lemm:exact_scaling_relation_for_quotients} implies that the first term is
\begin{equation*}
    \left[\frac{\mu^{\mathrm{H}}(Q^{\mathrm{L}})^{p}}{\mu^{\partial}(\delta Q^{\mathrm{L}})^{q}}\right]=2^{-\overline{\zeta}(p;q+1)}\mathbb{E}\left[\frac{\mu^{\mathrm{H}}(Q)^{p}}{\mu^{\partial}(I)^{q+1}}\right]<\infty
\end{equation*}
by assumption~\eqref{eq:cuteassumption}, and the second term is finite when $p<\frac{4}{\gamma^2}$ by Lemma~\ref{lemm:finiteness_of_quotients_with_positive_distance}.

When $p>1$, replace the subadditivity inequality by Minkowski's inequality for the $L_p$-norm which yields
\begin{equation*}
    \left[\frac{\mu^{\mathrm{H}}(\Gamma_{2^{-1}r})^{p}}{\mu^{\partial}(\partial \Gamma_{2^{-1}r})^{q+1}}\right]^{\frac{1}{p}}\leq \left[\frac{\mu^{\mathrm{H}}(Q^{\mathrm{L}})^{p}}{\mu^{\partial}(\partial \Gamma_{2^{-1}r})^{q+1}}\right]^{\frac{1}{p}}+\left[\frac{\mu^{\mathrm{H}}(Q^{\mathrm{U}})^{p}}{\mu^{\partial}(\partial \Gamma_{2^{-1}r})^{q+1}}\right]^{\frac{1}{p}}.
\end{equation*}
The same conclusion thus follows.
\end{proof}

We now go back to the covering of $Q_r$ by $\Gamma$-shaped regions $(\Gamma_{2^{-n}r})_{n\geq 1}$. When $0\leq p\leq 1$, by the subadditivity inequality,
\begin{equation}\label{eq:DejaVu_Gamma}
\begin{split}
    \mathbb{E}\left[\frac{\mu_r^{\mathrm{H}}(Q)^{p}}{\mu_r^{\partial}(I)^{q+1}}\right]&\leq\sum_{n=1}^{\infty}\mathbb{E}\left[\frac{\mu_r^{\mathrm{H}}(\Gamma_{2^{-n}r})^{p}}{\mu_r^{\partial}(I)^{q+1}}\right]\\
    &\leq\sum_{n=1}^{\infty}\mathbb{E}\left[\frac{\mu_r^{\mathrm{H}}(\Gamma_{2^{-n}r})^{p}}{\mu_r^{\partial}(\partial \Gamma_{2^{-n}r})^{q+1}}\right]\\
    &\leq\sum_{n=0}^{\infty}2^{-n\widetilde{\zeta}(p;q+1)}\mathbb{E}\left[\frac{\mu_r^{\mathrm{H}}(\Gamma_{2^{-1}r})^{p}}{\mu_r^{\partial}(\partial \Gamma_{2^{-1}r})^{q+1}}\right]
\end{split}
\end{equation}
where we used the scaling relation for the localized quotients of Lemma~\ref{lemm:exact_scaling_relation_for_localized_quotients}. When $p>1$, Minkowski's inequality for the $L_p$-norm yields
\begin{equation*}
    \mathbb{E}\left[\frac{\mu_r^{\mathrm{H}}(Q)^{p}}{\mu_r^{\partial}(I)^{q+1}}\right]^{\frac{1}{p}}\leq\sum_{n=0}^{\infty}2^{-\frac{n}{p}\widetilde{\zeta}(p;q+1)}\mathbb{E}\left[\frac{\mu_r^{\mathrm{H}}(\Gamma_{2^{-1}r})^{p}}{\mu_r^{\partial}(\partial \Gamma_{2^{-1}r})^{q+1}}\right]^{\frac{1}{p}}.
\end{equation*}

The expectation term is finite by Proposition~\ref{prop:cute_Gamma}, and the coefficient $\widetilde{\zeta}(p;q+1)$ is strictly positive if $p-\frac{q+1}{2}<\frac{2}{\gamma^2}-\frac{1}{2}$ and $q<2p-1$ by Remark~\ref{rema:sign-change}. Therefore, with $v=r$,
\begin{equation*}
    \mathbb{E}\left[\frac{\mu_r^{\mathrm{H}}(Q)^{p}}{\mu_r^{\partial}(I)^{q+1}}\right]<\infty
\end{equation*}
under the assumption~\eqref{eq:cuteassumption} and $q<2p-1, p<\frac{4}{\gamma^2}, p<\frac{2}{\gamma^2}+\frac{q}{2}$.

\subsubsection{The middle point case}
Suppose that $v=0$ in Lemma~\ref{lemm:cutecutelemma} and that
\begin{equation*}
    \mathbb{E}\left[\frac{\mu^{\mathrm{H}}(Q)^{p}}{\mu^{\partial}(I)^{q+1}}\right]<\infty.
\end{equation*}

Recall some notations that we already introduced in Section~\ref{sec:phenomenology_and_strategy_of_proofs} and denote by $\Pi_{2^{-n}r}=Q_{2^{-n+1}r}\setminus Q_{2^{-n}r}$ the $\Pi$-shaped part of order $2^{-n}$.

\begin{figure}[h]
\centering
\includegraphics[height=15em]{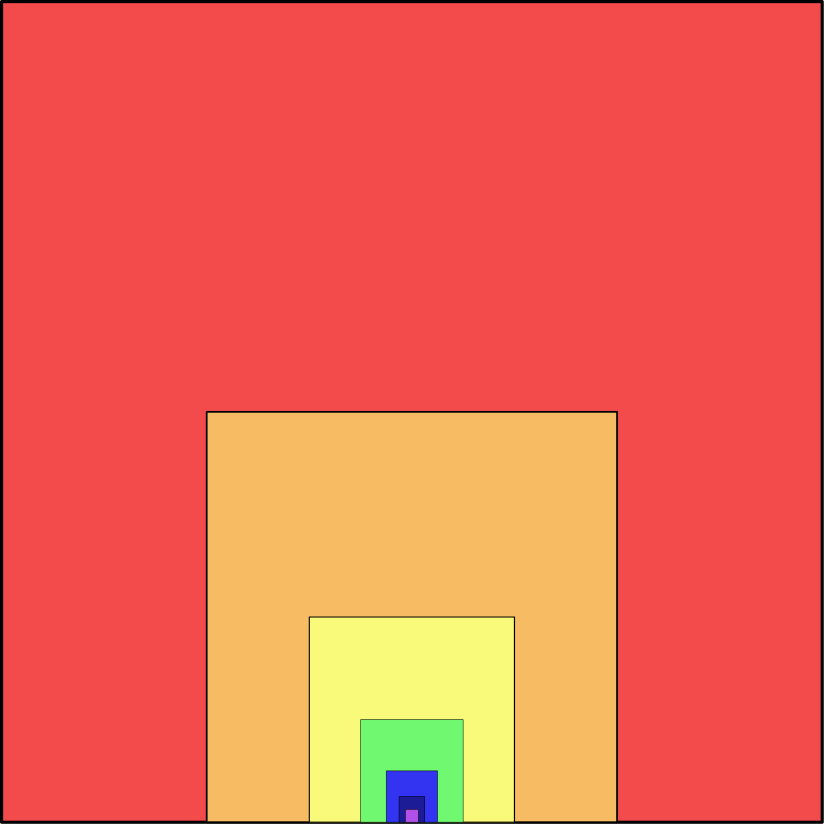}
\caption{Tiling a Carleson cube by $\Pi$-shaped regions (with respect to the middle point).}
\label{fig:Pi_Tiling}
\end{figure}

\begin{prop}[Finiteness of joint moments of $\Pi$-shaped regions]\label{prop:cute_Pi}
If $p<\frac{4}{\gamma^2}$ and~\eqref{eq:cuteassumption} holds. Then
\begin{equation}\label{eq:cute_Pi}
    \mathbb{E}\left[\frac{\mu^{\mathrm{H}}(\Pi_r)^{p}}{\mu^{\partial}(\delta\Pi_r)^{q+1}}\right]<\infty
\end{equation}
\end{prop}
\begin{proof}
The proof is similar to Proposition~\ref{prop:cute_Gamma} (by decomposing the $\Pi$-shaped region into $Q^{\mathrm{U}}$ and two Carleson cubes) and is omitted. 
\end{proof}

Recall that $(\Pi_{2^{-n}r}=Q_{2^{-n+1}r}\setminus Q_{2^{-n}r})_{n=0,1,\dots}$ is a tiling of $Q_r$. We proceed exactly as in the end point case:
\begin{itemize}
    \item If $0<p<1$, the subadditivity inequality yields
    \begin{equation*}
    \begin{split}
        \mathbb{E}\left[\frac{\mu_0^{\mathrm{H}}(Q)^{p}}{\mu_0^{\partial}(I)^{q+1}}\right]&\leq\sum_{n=0}^{\infty}\mathbb{E}\left[\frac{\mu_0^{\mathrm{H}}(\Pi_{2^{-n}r})^{p}}{\mu_0^{\partial}(I)^{q+1}}\right]\\
        &\leq\sum_{n=0}^{\infty}\mathbb{E}\left[\frac{\mu_0^{\mathrm{H}}(\Pi_{2^{-n}r})^{p}}{\mu_0^{\partial}(\delta \Pi_{2^{-n}r})^{q+1}}\right]\\
        &=\sum_{n=0}^{\infty}2^{-n\widetilde{\zeta}(p;q+1)}\mathbb{E}\left[\frac{\mu_0^{\mathrm{H}}(\Pi_{r})^{p}}{\mu_0^{\partial}(\delta \Pi_{r})^{q+1}}\right].
    \end{split}
    \end{equation*}
    \item If $p\geq 1$, Minkowski's inequality for the $L_p$-norm yields
    \begin{equation*}
        \mathbb{E}\left[\frac{\mu_0^{\mathrm{H}}(Q)^{p}}{\mu_0^{\partial}(I)^{q+1}}\right]^{\frac{1}{p}}\leq\sum_{n=0}^{\infty}2^{-\frac{n}{p}\widetilde{\zeta}(p;q+1)}\mathbb{E}\left[\frac{\mu_0^{\mathrm{H}}(\Pi_{r})^{p}}{\mu_0^{\partial}(\delta \Pi_{r})^{q+1}}\right]^{\frac{1}{p}}.
    \end{equation*}
\end{itemize}
Again, the expectation term is finite by Proposition~\ref{prop:cute_Gamma}, and the coefficient $\widetilde{\zeta}(p;q+1)$ is strictly positive if $p-\frac{q+1}{2}<\frac{2}{\gamma^2}-\frac{1}{2}$ by Remark~\ref{rema:sign-change} when $q+1<2p$. Therefore, with $v=0$,
\begin{equation*}
    \mathbb{E}\left[\frac{\mu_0^{\mathrm{H}}(Q)^{p}}{\mu_0^{\partial}(I)^{q+1}}\right]<\infty
\end{equation*}
under the assumption~\eqref{eq:cuteassumption} and $q<2p-1, p<\frac{4}{\gamma^2}, p<\frac{2}{\gamma^2}+\frac{q}{2}$.

\subsubsection{The general case}
Suppose that $v\geq 0$ in Lemma~\ref{lemm:cutecutelemma} (the case $v\leq 0$ is similar) and that
\begin{equation*}
    \mathbb{E}\left[\frac{\mu^{\mathrm{H}}(Q)^{p}}{\mu^{\partial}(I)^{q+1}}\right]<\infty.
\end{equation*}

\begin{figure}[h]
\centering
\includegraphics[height=15em]{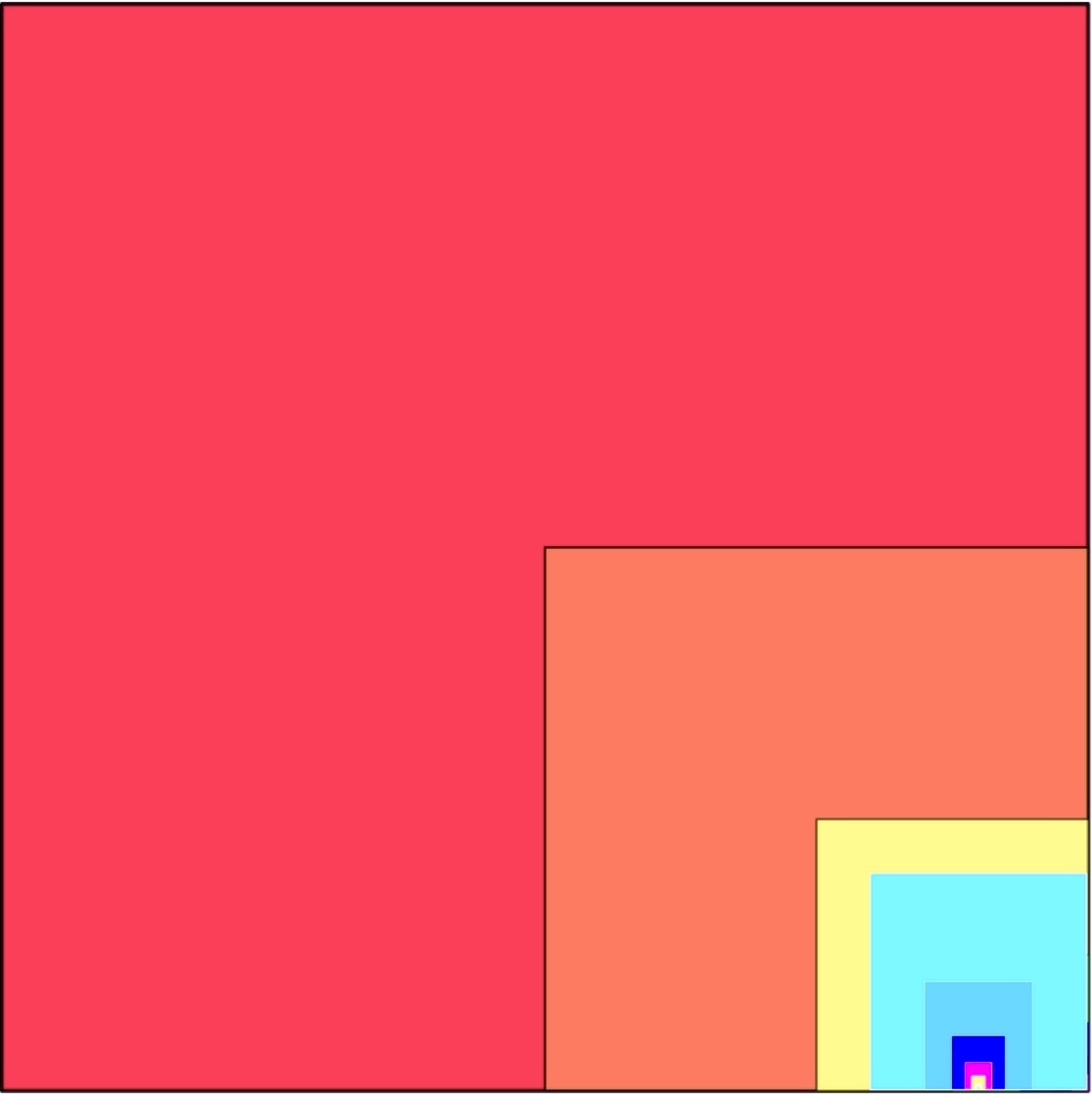}
\caption{Superposing a $\Gamma$-tiling and a rescaled $\Pi$-tiling with respect to any point on the boundary (the $\Pi$-tiling starts with the blue regions and converges to the point $v$ at the boundry).}
\label{fig:Pi+Gamma}
\end{figure}

Denote by $\rho=r-v$ the distance of $v$ to its closest end point $r$ and suppose that $l$ is the positive integer such that $2^{-l-1}r<\rho\leq 2^{-l}r$. Then the regions
\begin{equation*}
    Q'_{\rho}\quad\text{and}\quad \Gamma_{2^{-j}r},\quad j=0,1,\dots,l+1
\end{equation*}
cover $Q_r$ (see Figure~\ref{fig:Pi+Gamma} and~\eqref{eq:Q'_Definition} for the definition of $Q'$).

$\bullet$ Let us suppose first $0\leq p\leq 1$. We have the following control:
\begin{equation*}
    \mathbb{E}\left[\frac{\mu_v^{\mathrm{H}}(Q)^{p}}{\mu_v^{\partial}(I)^{q+1}}\right]\leq \mathbb{E}\left[\frac{\mu_v^{\mathrm{H}}(Q'_{\rho})^{p}}{\mu_v^{\partial}(I)^{q+1}}\right]+\mathbb{E}\left[\frac{\mu_v^{\mathrm{H}}(Q_r\setminus Q'_{\rho})^{p}}{\mu_v^{\partial}(I)^{q+1}}\right]\leq \mathbb{E}\left[\frac{\mu_v^{\mathrm{H}}(Q'_{\rho})^{p}}{\mu_v^{\partial}(\delta Q'_{\rho})^{q+1}}\right]+\mathbb{E}\left[\frac{\mu_v^{\mathrm{H}}(Q_r\setminus Q'_{\rho})^{p}}{\mu_v^{\partial}(I)^{q+1}}\right].
\end{equation*}
By Lemma~\ref{lemm:exact_scaling_relation_for_localized_quotients}, the first term in the above display is
\begin{equation*}
    (\frac{\rho}{r})^{\widetilde{\zeta}(p;q+1)}\mathbb{E}\left[\frac{\mu_0^{\mathrm{H}}(Q)^{p}}{\mu_0^{\partial}(I)^{q+1}}\right]\leq \mathbb{E}\left[\frac{\mu_0^{\mathrm{H}}(Q)^{p}}{\mu_0^{\partial}(I)^{q+1}}\right]
\end{equation*}
since $\rho\leq r$ and $\widetilde{\zeta}(p;q+1)>0$ if $q<2p-1$ and $p<\frac{2}{\gamma^2}+\frac{q}{2}$ by Remark~\ref{rema:sign-change}. The last expectation in the above display is finite under the condition $p<\frac{4}{\gamma^2}$, independent of $v\in[-r,r]$ by the discussion in the middle point case in the previous subsection. It thus remains to control the other term
\begin{equation*}
    \mathbb{E}\left[\frac{\mu_v^{\mathrm{H}}(Q_r\setminus Q'_{\rho})^{p}}{\mu_v^{\partial}(I)^{q+1}}\right]
\end{equation*}
above. At this stage need the following slight generalization of Proposition~\ref{prop:cute_Gamma}.

\begin{prop}\label{prop:Improved_cute_Gamma}
Suppose that~\eqref{eq:cuteassumption} holds. Then following supremum over the the localized $(p,q+1)$-joint moments is finite: 
\begin{equation}\label{eq:uniformly_cute_Gamma}
    \sup_{d\in[\frac{r}{4},\frac{r}{2}]}\mathbb{E}\left[\frac{\mu_r^{\mathrm{H}}(Q_r\setminus Q'_{d})^{p}}{\mu^{\partial}_r(\partial (Q_r\setminus Q'_{d}))^{q+1}}\right]<\infty
\end{equation}
if $q<2p-1, p<\frac{4}{\gamma^2}, p<\frac{2}{\gamma^2}+\frac{q}{2}$.
\end{prop}
\begin{proof}
The proof is almost identical to that of Proposition~\ref{prop:cute_Gamma}. Notice that (see Figure~\ref{fig:Improved_Gamma}) 
\begin{equation*}
    Q_r\setminus Q'_{d}\subset \left([-r,2r-2d]\times [0,r-2d]\right)\cup \left([-r,r]\times [\frac{r}{2},2r]\right).
\end{equation*}

\begin{figure}[h]
\centering
\includegraphics[height=15em]{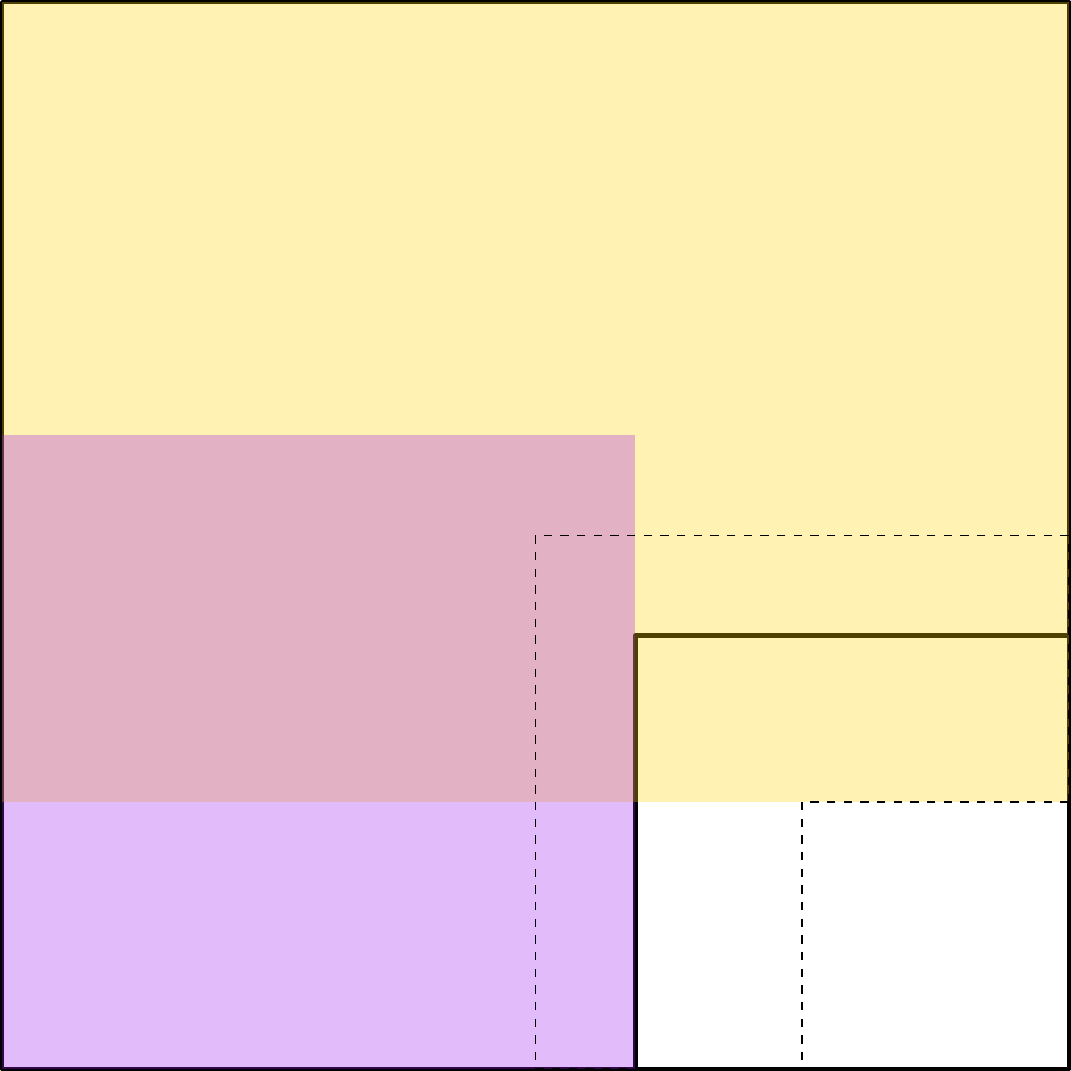}
\caption{Covering a $\Gamma$-shape region between two dyadic scales by a Carleson cube and a region far from the boundary.}
\label{fig:Improved_Gamma}
\end{figure}

Call the first set in the above union $S_1(d)$ and the second $S_2$ (corresponding respectively to the purple and the yellow regions in Figure~\ref{fig:Improved_Gamma}). Repeating the proof of Proposition~\ref{prop:Improved_cute_Gamma} shows that both
\begin{equation*}
    \mathbb{E}\left[\frac{\mu_r^{\mathrm{H}}(S_1)^{p}}{\mu^{\partial}_r([-r,2r-2d])^{q+1}}\right]\quad\text{and}\quad\mathbb{E}\left[\frac{\mu_r^{\mathrm{H}}(S_2)^{p}}{\mu^{\partial}_r([-r,0])^{q+1}}\right]
\end{equation*}
are finite under the conditions $q<2p-1, p<\frac{4}{\gamma^2}, p<\frac{2}{\gamma^2}+\frac{q}{2}$ and with an upper bound independent of $d$. This implies~\eqref{eq:uniformly_cute_Gamma} for $0\leq p\leq 1$ with the subadditivity inequality, and $p>1$ with Minkowski's inequality.
\end{proof}
We continue the above proof in the $0\leq p\leq 1$ case with this proposition. First, we need to compare
\begin{equation*}
    \mathbb{E}\left[\frac{\mu_v^{\mathrm{H}}(\Gamma_{2^{-j}r})^{p}}{\mu_v^{\partial}(\partial\Gamma_{2^{-j}r})^{q+1}}\right]\quad\text{and}\quad \mathbb{E}\left[\frac{\mu_r^{\mathrm{H}}(\Gamma_{2^{-j}r})^{p}}{\mu_r^{\partial}(\partial\Gamma_{2^{-j}r})^{q+1}}\right].
\end{equation*}
Notice that for any $0\leq j\leq l-1$, for any $z\in \Gamma_{2^{-j}r}$ and $w\in \partial\Gamma_{2^{-j}r}$, we have that
\begin{equation*}
    |z-v|\asymp 2^{-l}r\asymp |z-r|,\quad |w-v|\asymp 2^{-l}r\asymp |w-r|
\end{equation*}
where $\asymp$ means comparable within a multiplicative constant independent of $j$ and $l$. The same is true for $z\in \Gamma_{2^{-l}r}\cup \Gamma_{2^{-l-1}r}$ and $w\in \partial\Gamma_{2^{-l}r}\cup \partial\Gamma_{2^{-l-1}r}$ although the constant might change (but still independent of $l$).\footnote{In Figure~\ref{fig:Pi+Gamma}, the latter case means that $z$ is in the union of the yellow and the (partial) orange $\Gamma$-shaped regions, and $w$ is in the set of such $z$'s that are on the boundary $\mathbb{R}$.} This leads to the following: since $(\Gamma_{2^{-j}r})_{j=0,1,\dots,l+1}$ covers $Q_r\setminus Q'_{\rho}$, by the subadditivity inequality,
\begin{equation*}
\begin{split}
    \mathbb{E}\left[\frac{\mu_v^{\mathrm{H}}(Q_r\setminus Q'_{\rho})^{p}}{\mu_v^{\partial}(I)^{q+1}}\right]&\leq \sum_{j=0}^{l-1}\mathbb{E}\left[\frac{\mu_v^{\mathrm{H}}(\Gamma_{2^{-j}r})^{p}}{\mu_v^{\partial}(I)^{q+1}}\right]+\mathbb{E}\left[\frac{\mu_v^{\mathrm{H}}(\Gamma_{2^{-l}r}\cup\Gamma_{2^{-l-1}r}\setminus Q'_\rho)^{p}}{\mu_v^{\partial}(I)^{q+1}}\right]\\
    &\leq \sum_{j=0}^{l-1}\mathbb{E}\left[\frac{\mu_v^{\mathrm{H}}(\Gamma_{2^{-j}r})^{p}}{\mu_v^{\partial}(\partial\Gamma_{2^{-j}r})^{q+1}}\right]+\mathbb{E}\left[\frac{\mu_v^{\mathrm{H}}(\Gamma_{2^{-l}r}\cup\Gamma_{2^{-l-1}r}\setminus Q'_\rho)^{p}}{\mu_v^{\partial}(\partial(\Gamma_{2^{-l}r}\cup\Gamma_{2^{-l-1}r}\setminus Q'_\rho))^{q+1}}\right]\\
    &\lesssim \sum_{j=0}^{l-1}\mathbb{E}\left[\frac{\mu_r^{\mathrm{H}}(\Gamma_{2^{-j}r})^{p}}{\mu_r^{\partial}(\partial\Gamma_{2^{-j}r})^{q+1}}\right]+\mathbb{E}\left[\frac{\mu_r^{\mathrm{H}}(\Gamma_{2^{-l}r}\cup\Gamma_{2^{-l-1}r}\setminus Q'_\rho)^{p}}{\mu_r^{\partial}(\partial(\Gamma_{2^{-l}r}\cup\Gamma_{2^{-l-1}r}\setminus Q'_\rho))^{q+1}}\right]
\end{split}
\end{equation*}
where $\lesssim$ means within a multiplicative constant independent of $j$ and $l$. The first sum is finite because we have used it in the proof for the end point case $v=r$ in~\eqref{eq:DejaVu_Gamma}. The last expectation is finite because the region is a rescaled shape of $Q_r\setminus Q'_{d}$ for some $d\in[\frac{r}{4},\frac{r}{2}]$, we have
\begin{equation*}
    \mathbb{E}\left[\frac{\mu_r^{\mathrm{H}}(\Gamma_{2^{-l}r}\cup\Gamma_{2^{-l-1}r}\setminus Q'_\rho)^{p}}{\mu_r^{\partial}(\partial(\Gamma_{2^{-l}r}\cup\Gamma_{2^{-l-1}r}\setminus Q'_\rho))^{q+1}}\right]\leq 2^{-l\widetilde{\zeta}(p;q+1)}\sup_{d\in[\frac{r}{4},\frac{r}{2}]}\mathbb{E}\left[\frac{\mu_r^{\mathrm{H}}(Q_r\setminus Q'_{d})^{p}}{\mu^{\partial}_r(\partial (Q_r\setminus Q'_{d}))^{q+1}}\right]
\end{equation*}
where the scaling relation comes from Lemma~\ref{lemm:exact_scaling_relation_for_localized_quotients}. Since $\widetilde{\zeta}(p;q+1)>0$ under our assumptions that $p<\frac{2}{\gamma^2}+\frac{q}{2}$ and $q<2p-1$ and by Proposition~\ref{prop:Improved_cute_Gamma}, the last expectation in the above display is finite if furthermore $p<\frac{4}{\gamma^2}$. Therefore
\begin{equation*}
    2^{-l\widetilde{\zeta}(p;q+1)}\sup_{d\in[\frac{r}{4},\frac{r}{2}]}\mathbb{E}\left[\frac{\mu_r^{\mathrm{H}}(Q_r\setminus Q'_{d})^{p}}{\mu^{\partial}_r(\partial (Q_r\setminus Q'_{d}))^{q+1}}\right]\leq \mathbb{E}\left[\frac{\mu_r^{\mathrm{H}}(Q_r\setminus Q'_{d})^{p}}{\mu^{\partial}_r(\partial (Q_r\setminus Q'_{d}))^{q+1}}\right]<C
\end{equation*}
if $p<\min(\frac{2}{\gamma^2}+\frac{q}{2},\frac{4}{\gamma^2})$ and $q<2p-1$. This finishes the proof in the $0\leq p\leq 1$ case since all multiplicative constants in the notations $\asymp$ and $\lesssim$ and final upper bounds are independent of $l$, thus independent of $v$.

$\bullet$ If $p>1$, we replace the subadditivity inequality by Minkowski's inequality for the $L_p$-norm as before, and the same conclusion holds.

\bigskip

To summarize, we have proven that the integrand in~\eqref{eq:deadlyintegral} is finite for all $v\in[-r,r]$, under the conditions $p<\min(\frac{2}{\gamma^2}+\frac{q}{2},\frac{4}{\gamma^2})$ and $q<2p-1$, with a uniform upper bound over $v$ depending only on the unlocalized $(p,q)$-joint moment $\mathbb{E}\left[\frac{\mu^{\mathrm{H}}(Q)^{p}}{\mu^{\partial}(I)^{q+1}}\right]$. This finishes the proof of Lemma~\ref{lemm:cutecutelemma} and Theorem~\ref{th:main_result_positive}.

\subsection{A slight generalization to non-tangential domains}
In this last section, we make a small observation that will be useful in the following articles of this series. This consists of changing the shape of the Carleson cube to a trapezoid. Let us call a domain $T$ a (boundary) non-tangential trapezoid of the upper-half plane $\mathbb{H}$ if $T$ is the interior of the lines
\begin{equation*}
    \{y=0\}, \{y=2r\}, \{y=a+k_1x\}, \{y=b+k_2x\}
\end{equation*}
with $a<b$, $k_1,k_2\in\mathbb{R}$ and $r$ small enough such that $T$ has the shape of a trapezoid. The interesting case for us is when the intersection $\{y=a+k_1x\}\cap \{y=b+k_2x\}$ is not in the upper-half plane $\mathbb{H}$.

We claim that the proofs above generalize to the following result:
\begin{prop}[Joint bulk/boundary moment bounds for non-tangential trapezoids]
Let $T$ be a small enough non-tangential trapezoid near the origin of $\mathbb{H}$ and $I$ its intersection with $\mathbb{R}$. Suppose that $p,q\geq 0$. Then
\begin{equation*}
    \mathbb{E}\left[\frac{\mu^{\mathrm{H}}(T)^{p}}{\mu^{\partial}(I)^{q}}\right]<\infty
\end{equation*}
if $p<\min(\frac{2}{\gamma^2}+\frac{q}{2},\frac{4}{\gamma^2})$.
\end{prop}
Indeed, suppose that $k_1\leq 0$ and $k_2\geq 0$ so that $T$ is larger than the Carleson cube $Q$ with boundary $I$ (if not, $T$ is contained in such a configuration and the situation only gets easier). All the scaling relations are the same, and when we replace a Carleson cube $Q$ by this kind of trapezoid $T$, instead of having an exact tiling we get some overlapping between the tiles, which offers a slightly worse upper bound but is still finite by the same arguments.

\section{Discussions and perspectives}\label{sec:discussions_and_perspectives}
In this note, we demonstrated how to implement the localization trick by constructing an auxiliary Gaussian multiplicative chaos measure that correctly captures the singular behavior of the original random measure. We list some further questions that should be addressed in this direction:
\begin{enumerate}
    \item Optimality of the joint moment threshold. As we discussed several times in the article, we believe that the threshold on the finiteness of the joint moment of bulk/boundary quotients given by Theorem~\ref{th:main_result_positive} is optimal, and there is explosion at this threshold. It would seem that a slight modification of our method is enough to at least prove the optimality of the threshold, but we are unable to prove this for now. Another possible route is to first study the right tail profile problem of the bulk/boundary quotient (see questions below), as the proofs about the right tail profiles of classical Gaussian multiplicative chaos measures do not use directly the explosion of the moments. However, the drawback of their method is that one has to pin down very exactly the contribution of the bulk/boundary measures to the right tail profile, and we think that there are softer methods to show the explosion of the joint moment without performing this finer analysis.
    \item The critical bulk/boundary quotient. It is interesting to note that in the case where $q=2p$, the ratio
    \begin{equation*}
        R=\frac{\mu^{\mathrm{H}}(Q)}{\mu^{\partial}(I)^{2}}
    \end{equation*}
    plays a special role in boundary Liouville conformal field theory (see formulas in~\cite[Section~3.6]{huang2018liouville}), and establishing moment bounds for $R$ is of physical importance. Baptiste Cerclé told us that conjectural exact formulas on $R$ exist and were communicated to him in private by physicists. According to Cerclé, the formulas exhibit singularities for $R^{p}$ with $p=\frac{4}{\gamma^2}$, confirming the moment bound in the critical case is $p_c=\frac{4}{\gamma^2}$, see Theorem~\ref{th:main_result_positive}. We hope that the current note will allow one to proceed to the next step, which would be to establish an exact formula for the law of $R$ following similar methodologies of~\cite{Rhodes_2019,Kupiainen_2020,Remy_2020,remy2022integrability,ang2021fzz}.
    \item Tail profile of bulk/boundary quotient. Given the threshold in Theorem~\ref{th:main_result_positive} (and suppose that it is optimal), the following right tail profiles should be expected for the bulk/boundary quotient in the case $p,q>0$:
    \begin{equation*}
        \mathbb{P}\left[\frac{\mu^{\mathrm{H}}(Q)^{p}}{\mu^{\partial}(I)^{q}}>t\right]\sim C_0 t^{-\frac{1}{\max(2p-q,p)}\frac{4}{\gamma^2}}.
    \end{equation*}
    The tail profile constants for $q=0$ will be proved in this series of papers, and we expect the method to be robust enough to treat the general non-critical case. It would be interesting to see if $C_0$ is related to some other bulk/boundary quotients. In fact, showing this right tail profile will settle the problem about the explosion at criticality of the joint bulk/boundary moments in Theorem~\ref{th:main_result_positive} (we think that this is possible as the proofs of the tail profile of classical Gaussian multiplicative chaos measures in~\cite{Rhodes_2019,wong2020universal} do not use the explosion of moments at criticality), but as stated before, we suspect that there are softer ways to go around.
\end{enumerate}

\bibliographystyle{alpha}

\end{document}